\documentclass{article}
\usepackage{amsmath,amssymb,amsthm}
\usepackage{graphicx,subfigure,float,url,color}
\usepackage{mathrsfs}
\usepackage[colorlinks=true]{hyperref}
\usepackage{pdfsync}
\usepackage{enumitem}

\graphicspath{{fig/}}

\def\R{\mathrm{I\kern-0.21emR}}
\def\N{\mathrm{I\kern-0.21emN}}

\newcommand{\C} {\mathbb{C}}

\renewcommand{\geq}{\geqslant}
\renewcommand{\leq}{\leqslant}

\newtheorem{theorem}{Theorem}  
\newtheorem{proposition}{Proposition}
\newtheorem{corollary}{Corollary}

\newtheorem{lemma}{Lemma}

\theoremstyle{definition}\newtheorem{remark}{Remark}

\newcommand{\xmin}{x_{\mathrm{min}}}
\newcommand{\tmin}{t_{\mathrm{min}}}
\newcommand{\umin}{u_{\mathrm{min}}}
\newcommand{\umax}{u_{\mathrm{max}}}
\newcommand{\omin}{\omega_{\mathrm{min}}}
\newcommand{\omax}{\omega_{\mathrm{max}}}

\title{A meaningful optimal control problem in quantum and classical physics}

\author{
Arnaud Lazarus\footnote{Sorbonne Universit\'e, CNRS, Institut Jean Le Rond d'Alembert, F-75005 Paris, France, also affiliated with Massachusetts Institute of Technology, Department of Mathematics, Cambridge, MA 02139, USA (\texttt{arnaud.lazarus@sorbonne-universite.fr})).}
\and
Emmanuel Tr\'elat\footnote{Sorbonne Universit\'e, Universit\'e Paris Cit\'e, CNRS, Inria, Laboratoire Jacques-Louis Lions, LJLL, F-75005 Paris, France (\texttt{emmanuel.trelat@sorbonne-universite.fr}).}
}

\date{}

\begin{document}

\maketitle

\begin{abstract}
In this paper we study and solve an optimal control problem motivated by applications in quantum and classical physics. Although apparently simple, this optimal control problem is not easy to solve and we resort to various elaborated methods of optimal control theory. We finally show its relationships to two problems in physics: the computation of the ground state for 1D Schr\"odinger operators with a finite potential well, and the optimal dynamical Kapitza stabilization problem. 
\end{abstract}

\section{The optimal control problem}\label{sec_OCP}

Given any $\umin,\umax\in\R$ such that $\umin<0<\umax$, given any $T>0$, we consider the optimal control problem 
\begin{subequations}\label{OCP}
\begin{align}
& \dot x(t) = y(t) , \label{OCP_dyn_x} \\
& \dot y(t) = -u(t) x(t)  , \label{OCP_dyn_y} \\
& \umin \leq u(t)\leq \umax , \label{OCP_cont_const} \\
& \min \int_0^T u(t)\, dt , \label{OCP_cost} 
\end{align}
\end{subequations}
where \eqref{OCP_dyn_x}, \eqref{OCP_dyn_y} and \eqref{OCP_cont_const} are written for almost every $t\in[0,T]$, with the periodicity conditions and nontriviality constraint
\begin{subequations}\label{OCP_0}
\begin{align}
& x(0)=x(T), \quad y(0)=y(T),  \label{OCP_term} \\
& x(0)^2+y(0)^2>0 .  \label{OCP_>0}
\end{align}
\end{subequations}
%Throughout the article, we denote $q=(x,y)^\top$, so that \eqref{OCP_term} is written as the periodicity condition $q(0)=q(T)$. 
The (non-closed) constraint \eqref{OCP_>0} ensures nontriviality of optimal solutions, if they exist. Indeed, if we remove the constraint \eqref{OCP_>0}, then obviously the unique optimal trajectory is $x(\cdot)=y(\cdot)=0$, $u(\cdot)=\umin$, and the optimal value is $T\umin$.

\begin{theorem}\label{main_thm}
Let $T>0$ be arbitrary.
There exists a unique optimal solution $(x(\cdot),y(\cdot),u(\cdot))$ of the optimal control problem \eqref{OCP}-\eqref{OCP_0}, satisfying
\begin{equation}\label{OCP_per_0}
x(0)=x(T)=1, \quad y(0)=y(T)=0 ,
\end{equation}
and
\begin{equation}\label{0<x<=1}
0<x(t)\leq 1\qquad\forall t\in[0,T].
\end{equation}
Extending $(x(\cdot),y(\cdot),u(\cdot))$ to the whole $\R$ by $T$-periodicity, any other optimal solution of \eqref{OCP}-\eqref{OCP_0} is given by $(\mu x(\cdot+\delta),\mu y(\cdot+\delta),u(\cdot+\delta))$ for some $\mu\neq 0$ and $\delta\in\R$ (i.e., homothety and shifting in time).

The optimal control $u(\cdot)$ is bang-bang with two switchings:
\begin{equation}\label{u_opt}
u(t) = \left\{ \begin{array}{ll}
\umax & \textrm{if}\ \ 0\leq t<t_1 , \\
\umin & \textrm{if}\ \ t_1<t<T-t_1 , \\
\umax & \textrm{if}\ \ T-t_1<t\leq T ,
\end{array}\right.
\end{equation}
where the switching time $t_1 \in \Big(0, \frac{1}{\omax}\mathrm{Arctan}\big(\frac{\omin}{\omax}\big) \Big)$ is characterized by the bijective relation
\begin{equation}\label{Tt1}
T = \frac{1}{\omin} \ln \left( \frac{1+\frac{\omax}{\omin}\tan(\omax t_1)}{1-\frac{\omax}{\omin}\tan(\omax t_1)} \right) +2t_1
\end{equation}
($T$ is an increasing function of $t_1$) with
\begin{equation}\label{def_omega}
\omin = \sqrt{-\umin},\quad\omax = \sqrt{\umax}.
\end{equation}
The optimal trajectory $(x(\cdot),y(\cdot))$ is symmetric with respect to the $x$-axis and is homeomorphic to a clockwise circle (see Figure \ref{fig_thm}), entirely contained in the half-plane $x>0$, consisting of the concatenation of an arc of ellipse and of an arc of hyperbole. Precisely, we have
\begin{equation}\label{xopt}
x(t) = \left\{ \begin{array}{ll}
\cos(\omax t) & \textrm{if}\ \ 0\leq t\leq t_1 , \\[2mm]
\frac{1}{2}\big( \cos(\omax t_1)-\frac{\omax}{\omin}\sin(\omax t_1) \big) e^{\omin(t-t_1)} & \\[1mm]
\qquad\quad + \frac{1}{2}\big( \cos(\omax t_1)+\frac{\omax}{\omin}\sin(\omax t_1) \big) e^{-\omin(t-t_1)} \ \ & \textrm{if}\ \ t_1\leq t\leq T-t_1 , \\ [2mm]
\cos(\omax (T-t)) & \textrm{if}\ \ T-t_1\leq t\leq T .
\end{array}\right.
\end{equation}
The cost of the optimal trajectory is
\begin{equation}\label{cost}
\begin{split}
\int_0^T u(t)\, dt &= 2t_1\umax+(T-2t_1)\umin  \\
&= -\omin\ln \left( \frac{1+\frac{\omax}{\omin}\tan(\omax t_1)}{1-\frac{\omax}{\omin}\tan(\omax t_1)} \right) + 2 \omax^2 t_1 <0.
\end{split}
\end{equation}
It is always negative.
\end{theorem}

\begin{figure}[h]
%\begin{center}
%\resizebox{15cm}{!}{\input fig_thm.pdf_t}
%\end{center}
\centerline{\includegraphics[width=9cm]{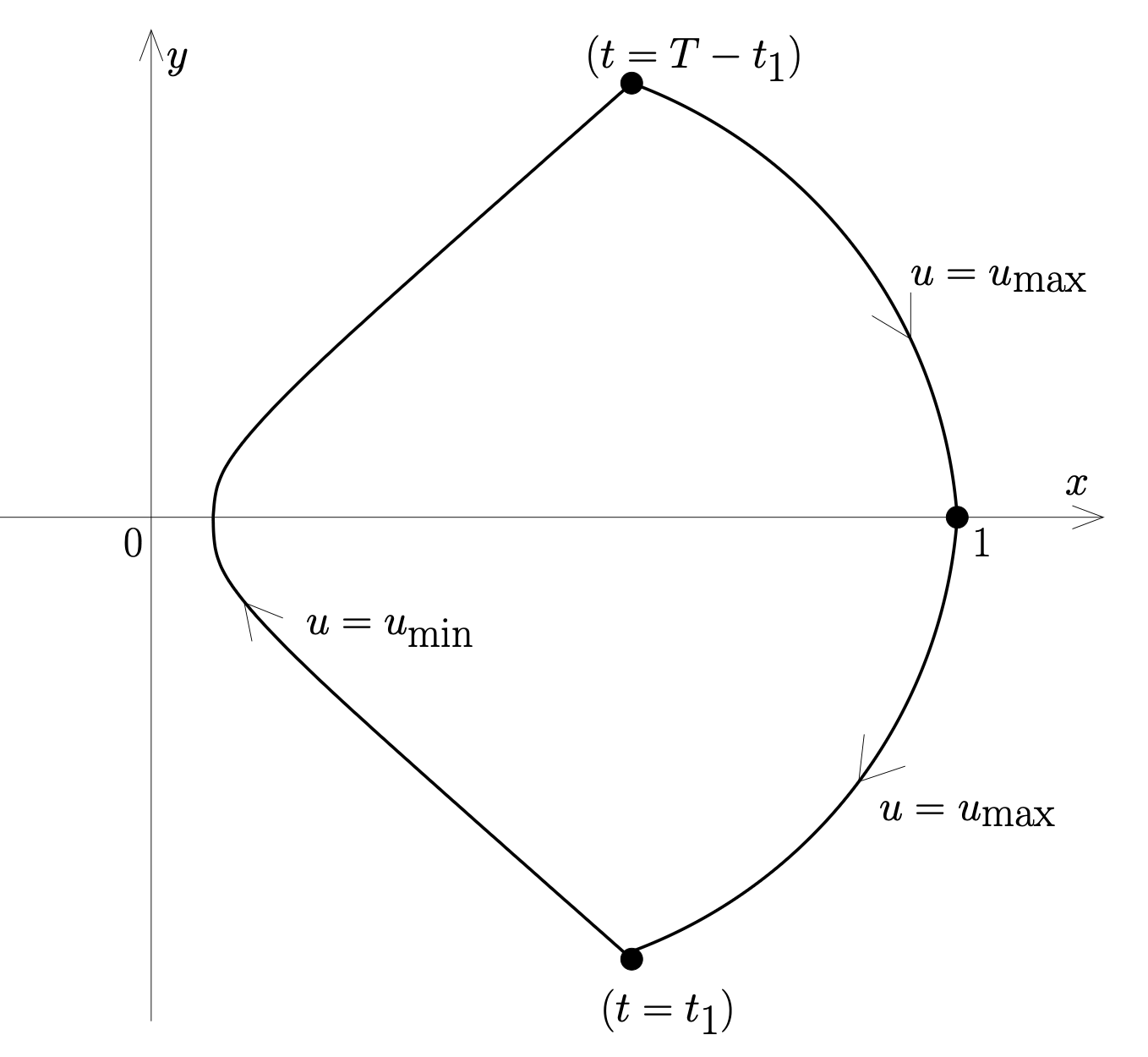}}
\caption{Optimal trajectory (Theorem \ref{main_thm}).}\label{fig_thm}
\end{figure}

\begin{remark}\label{rem1}
The optimal trajectory starts (and ends) at $(1,0)$ with a vertical tangent. On $[0,t_1]\cup[T-t_1,T]$, the curve $t\mapsto(x(t),y(t))$ follows an ellipse of equation $x^2+\frac{y^2}{\omax^2}=1$. The value $x(0)=x(T)=1$ is the maximal value of $x(t)$ as $t\in[0,T]$.

On $[t_1,T-t_1]$, the curve $t\mapsto(x(t),y(t))$ follows a hyperbole of equation $x^2-\frac{y^2}{\omin^2} = c(t_1)$ where $c(t_1)=\cos^2(\omax t_1)\Big( 1-\frac{\omax^2}{\omin^2}\tan^2(\omax t_1)\Big)>0$.
We have $y(\frac{T}{2})=0$, and $x(\frac{T}{2})=\sqrt{c(t_1)}$ is the minimal value of $x(t)$. 

The relation \eqref{Tt1} gives $T$ as a function of $t_1$, which is increasing (see Lemma \ref{lem_Tt1} further).
When $T\rightarrow+\infty$, we have:
\begin{itemize}
\item $\displaystyle t_1\longrightarrow\frac{1}{\omax}\mathrm{Arctan}\Big(\frac{\omin}{\omax}\Big)$.
\item $\displaystyle x(t_1)=\cos(\omax t_1)\longrightarrow\frac{\omax}{\sqrt{\omin^2+\omax^2}} = \sqrt{\frac{\umax}{\umax-\umin}}$.
\item $\displaystyle y(t_1)=-\omax\sin(\omax t_1) \longrightarrow -\frac{\omin\omax}{\sqrt{\omin^2+\omax^2}} = -\sqrt{\frac{-\umin\umax}{\umax-\umin}}$.
\item The curve $t\mapsto(x(t),y(t))$, restricted to $\big[t_1,\frac{T}{2}\big]$ (along which $u(t)=\umin$), converges to the segment joining the point $\Big( \sqrt{\frac{\umax}{\umax-\umin}}, -\sqrt{\frac{-\umin\umax}{\umax-\umin}} \Big)$ to the point $(0,0)$.
Moreover, for $t\in\big[t_1,\frac{T}{2}\big]$, we have
$$
x(t) \sim  \frac{\omax}{\sqrt{\omin^2+\omax^2}} \exp\bigg( \frac{\omin}{\omax} \mathrm{Arctan}\Big(\frac{\omin}{\omax}\Big) \bigg) \left( e^{\omin(t-T)}+e^{-\omin t} \right)
$$
In particular, 
$$
x\Big(\frac{T}{2}\Big) \sim \frac{2\omax}{\sqrt{\omin^2+\omax^2}} \exp\bigg( \frac{\omin}{\omax} \mathrm{Arctan}\Big(\frac{\omin}{\omax}\Big) -\omin\frac{T}{2} \bigg) .
$$
\item The cost of the trajectory satisfies $\int_0^T u(t)\, dt \sim T\umin < 0$.
\end{itemize}
Actually, when $T\rightarrow+\infty$, the curve $(x(\cdot),y(\cdot))$ converges to a ``teardrop", symmetric with respect to the $x$-axis, consisting of a ``V" (two segments) rotated by $-\pi/2$, with edge at the origin, completed with an arc of ellipse. We refer to Section \ref{sec_turnpike} for more comments on this manifestation of the well known turnpike phenomenon. 
\end{remark}

%Note that the control system \eqref{OCP_dyn_x}-\eqref{OCP_dyn_y} is control-affine, of the form
%$$
%\dot q(t) = f_0(q(t)) + u(t)f_1(q(t)) ,
%$$
%with the two vector fields $f_0$ and $f_1$ defined by $f_0(q)=y\,\partial_x$ and $f_1(q)=-x\,\partial_y$.

\section{Proof of Theorem \ref{main_thm}}
\subsection{Preliminaries and first reduction of the problem}\label{sec_firstreduction}
\paragraph{Existence of an optimal solution.}
We claim that there exists at least one optimal trajectory $(x(\cdot),y(\cdot),u(\cdot))$ solution of \eqref{OCP}-\eqref{OCP_0}. 

Indeed, we preliminary note that the (constant) trajectory defined by $x(t)=1$, $y(t)=0$ and $u(t)=0$ for every $t\in[0,T]$ is a solution of \eqref{OCP_dyn_x}-\eqref{OCP_dyn_y}-\eqref{OCP_cont_const}-\eqref{OCP_term}-\eqref{OCP_>0}, hence the set of admissible trajectories is nonempty. Now, the existence of an optimal solution follows from standard existence results (see \cite[Theorem 2.9]{Trelat_bookSB} or see \cite{Cesari, LeeMarkus}), noting that the control system \eqref{OCP_dyn_x}-\eqref{OCP_dyn_y} is control-affine, that the cost functional \eqref{OCP_cost} is convex, and that the controls are bounded (by \eqref{OCP_cont_const}).

\paragraph{Preliminary remarks.} We start with some easy remarks.

\begin{lemma}\label{lem_prelim}
Given any optimal solution $(x(\cdot),y(\cdot),u(\cdot))$ of \eqref{OCP}-\eqref{OCP_0}:
\begin{enumerate}[label=$\bf (A_{\theenumi})$]
\item\label{A1} We have $x(t)^2+y(t)^2>0$ for every $t\in[0,T]$.
\item\label{A2} Extending $(x(\cdot),y(\cdot),u(\cdot))$ to the whole $\R$ by $T$-periodicity, for any $\delta\neq 0$, $(x(\cdot+\delta),y(\cdot+\delta),u(\cdot+\delta))$ (translation in time) is also an optimal solution of \eqref{OCP}-\eqref{OCP_0}. %, with the same optimal value. 
\item\label{A3} For any $\mu\neq 0$, $(\mu x(\cdot),\mu y(\cdot),u(\cdot))$ (homothety on $q$) is also an optimal solution of \eqref{OCP}-\eqref{OCP_0}. %, with the same optimal value. 
\item\label{A4} $x(\cdot)$ is nontrivial, $\dot x(\cdot)$ is continuous on $[0,T]$ and $\dot x(0)=\dot x(T)$. 
\end{enumerate}
\end{lemma}

\begin{proof}
If $(x(\cdot),y(\cdot))$ passes through $(0,0)$ then it must remain at $(0,0)$ for every time, by Cauchy uniqueness, which contradicts \eqref{OCP_>0}. This gives \ref{A1}.
Now, \ref{A2} is obtained by using \ref{A1}, and \ref{A3} is obvious. It remains to establish \ref{A4}. By contradiction, if $x(\cdot)=0$, then \eqref{OCP_dyn_x} implies that $y(\cdot)=0$, contradicting \eqref{OCP_>0}. Hence $x(\cdot)$ is nontrivial. The second part of \ref{A4} follows from the facts that $\dot x(t)=y(t)$ by \eqref{OCP_dyn_x} and that $y(\cdot)$ is continuous and $T$-periodic by \eqref{OCP_per}.
\end{proof}

\begin{remark}
According to \ref{A4}, $x(\cdot)$ is $C^1$ and $T$-periodic. 
In contrast, $y(\cdot)$ is not $C^1$ on $[0,T]$. Indeed, $\dot y(t)=-u(x)x(t)$ by \eqref{OCP_dyn_y} and $u$ is not continuous, as it will be proved further.
\end{remark}

\paragraph{First reduction of the problem.} 
Combining \ref{A2}, \ref{A3} and \ref{A4} of Lemma \ref{lem_prelim}, without loss of generality we can replace \eqref{OCP_term} and \eqref{OCP_>0} by 
\begin{equation}\label{OCP_per}
x(0)=x(T)=1, \quad y(0)=y(T).
\end{equation}
Hence, in what follows we consider the optimal control problem \eqref{OCP}-\eqref{OCP_per}.

Note that, at this step, $y(0)=y(T)$ is let free. We will prove further that any solution of the optimal control problem \eqref{OCP}-\eqref{OCP_per} satisfies $y(0)=y(T)=0$, i.e., \eqref{OCP_per_0} is satisfied. 
Further, we will also perform a second reduction to arrive at the state constraint $x(t)\leq 1$ (i.e., half of \eqref{0<x<=1}), and establish the other half of \eqref{0<x<=1}.

\medskip
The proof goes in several steps, by first applying the Pontryagin maximum principle and then establishing various properties. 
It turns out that the proof is far from being easy and does not follow straigthforwardly from the Pontryagin maximum principle, as one could suspect at the first glance. This difficulty is probably due to the existence of too many symmetries and geometric transforms, that make the extremal equations, in some sense, somewhat degenerate. We will even have to resort to the Stokes theorem applied with a nonclassical one-differential form (different from the more classical clock form), thus employing arguments that are of a global nature. This is not so common in the study of optimal control problems.

\subsection{Application of the Pontryagin maximum principle}\label{sec_PMP}
The Hamiltonian of the optimal control problem \eqref{OCP} is
$$
H(x,y,p_x,p_y,p^0,u) = p_xy-p_yxu+p^0u .
$$
Given any optimal solution $(x(\cdot),y(\cdot),u(\cdot))$ of \eqref{OCP}-\eqref{OCP_per} on $[0,T]$, by the Pontryagin maximum principle (see \cite{LeeMarkus, Pontryagin, Trelat_bookSB}), there exist an absolutely continuous \emph{adjoint vector} $(p_x(\cdot),p_y(\cdot))$ on $[0,T]$ and $p^0\leq 0$ such that $(p_x(\cdot),p_y(\cdot),p^0)\neq(0,0,0)$ and
\begin{subequations}\label{adjoint}
\begin{align}
& \dot p_x(t) = u(t)p_y(t) \label{adjoint_x} \\
& \dot p_y(t) = -p_x(t)  \label{adjoint_y}
\end{align}
\end{subequations}
and
\begin{equation}\label{maximization}
H(x(t),y(t),p_x(t),p_y(t),p^0,u(t)) = \max_{\umin\leq v\leq\umax}H(x(t),y(t),p_x(t),p_y(t),p^0,v)
\end{equation}
for almost every $t\in[0,T]$.
Defining the (absolutely continuous) \emph{switching function}
\begin{equation}\label{def_phi}
\varphi(t) = -p_y(t)x(t)+p^0 \qquad\forall t\in[0,T],
\end{equation}
the maximization condition \eqref{maximization} gives
$$
\varphi(t) u(t) = \max_{\umin\leq v\leq\umax} (\varphi(t) v) ,
$$
for almost every $t\in[0,T]$, which yields
\begin{equation}\label{u_extr}
u(t) = \left\{ \begin{array}{ll}
\umin & \textrm{if}\ \ \varphi(t)<0, \\
\umax & \textrm{if}\ \ \varphi(t)>0,
\end{array}\right.
\end{equation}
for almost every $t\in[0,T]$.
At this step, we state nothing on the closed subset $I$ of $[0,T]$ where the continuous function $\varphi$ vanishes identically. This set could be of positive measure and have a complicated structure. Actually, we will prove further in Lemma \ref{lem_nonsing} (Section \ref{sec_nonsing}) that $I$ is of Lebesgue measure zero and thus \eqref{u_extr} is enough to fully describe the optimal control almost everywhere. As a consequence of Lemma \ref{lem_nonsing}, since $\varphi$ is continuous, the optimal control $u(\cdot)$ is \emph{bang-bang}, i.e., the time interval $[0,T]$ is a countable union of open intervals along which either $u(t)=\umin$ or $u(t)=\umax$. But this result is far from being obvious and to prove we will first establish a number of other results.

For now, let us first finish to apply the Pontryagin maximum principle, which also gives the following additional information.
The maximized Hamiltonian defined by
\begin{equation}\label{def_H1}
H_1(x(t),y(t),p_x(t),p_y(t),p^0) %= \max_{\umin\leq v\leq\umax}H(x(t),y(t),p_x(t),p_y(t),p^0,v)
= p_x(t)y(t) + \max_{\umin\leq v\leq\umax}(v\varphi(t))
%= \umax\frac{\varphi(t)+\vert\varphi(t)\vert}{2} + \umin\frac{\varphi(t)-\vert\varphi(t)\vert}{2}
\end{equation}
for every $t\in[0,T]$ is constant on $[0,T]$.
Moreover, by the \emph{transversality conditions} of the Pontryagin maximum principle, the periodicity condition $y(0)=y(T)$ (whose value is let free) of \eqref{OCP_per} implies that 
\begin{equation}\label{py_per}
p_y(0)=p_y(T)
\end{equation}
(see \cite[Section 2.2.3]{Trelat_bookSB}), i.e., that $p_y(\cdot)$ is $T$-periodic. Actually, we are going to see that $p_x(\cdot)$ is $T$-periodic as well (see Lemma \ref{lem_per} further).

%In what follows we denote
%$$
%q = \begin{pmatrix} x\\ y\end{pmatrix}
%\qquad\textrm{and}\qquad
%p = \begin{pmatrix} p_x\\ p_y\end{pmatrix}.
%$$
Recall that $(x(\cdot),y(\cdot), p_x(\cdot),p_y(\cdot), p^0, u(\cdot))$ is called an \emph{extremal lift} of the optimal trajectory $(x(\cdot),y(\cdot),u(\cdot))$. The triple $(p_x(T),p_y(T), p^0)$ is defined up to scaling. The extremal is said to be \emph{normal} if $p^0\neq 0$, and in this case it is usual to normalize it so that $p^0=-1$. It is said to be \emph{abnormal} if $p^0=0$. 

We next exploit the various conditions given by the Pontryagin maximum principle.

\subsection{First properties}
We start with an easy lemma.

\begin{lemma}\label{lem1}
Given any optimal solution $(x(\cdot),y(\cdot),u(\cdot))$ of \eqref{OCP}-\eqref{OCP_per} and given any extremal lift $(x(\cdot),y(\cdot), p_x(\cdot),p_y(\cdot), p^0, u(\cdot))$ of it,
the function $t\mapsto p_x(t)x(t)+p_y(t)y(t)$ is constant on $[0,T]$.
\end{lemma}

\begin{proof}
The result obviously follows from the fact that
$$
\frac{d}{dt} ( p_x(t)x(t)+p_y(t)y(t) ) = 0 ,
$$
which is inferred from \eqref{OCP_dyn_x}, \eqref{OCP_dyn_y}, \eqref{adjoint_x} and \eqref{adjoint_y}.
\end{proof}

Note that Lemma \ref{lem1} is valid independently on the constraints on the initial and final conditions. 

\begin{lemma}\label{lem_per}
Given any optimal solution $(x(\cdot),y(\cdot),u(\cdot))$ of \eqref{OCP}-\eqref{OCP_per} and given any extremal lift $(x(\cdot),y(\cdot), p_x(\cdot),p_y(\cdot), p^0, u(\cdot))$ of it,
we have 
$$
p_x(0)=p_x(T), \quad p_y(0)=p_y(T) .
$$
In other words, the adjoint vector is $T$-periodic.
\end{lemma}

\begin{proof}
We infer from Lemma \ref{lem1} that
$$
p_x(0)x(0)+p_y(0)y(0) = p_x(T)x(T)+p_y(T)y(T) .
$$
Since $x(0)=x(T)=1$ (by \eqref{OCP_per}) and $p_y(0)=p_y(T)$ (by \eqref{py_per}), the conclusion follows.
\end{proof}

\subsection{Second reduction of the problem}\label{sec_secondreduction}

\begin{lemma}\label{lem_y0}
Shifting in time and using an homothety if necessary, without loss of generality, we can assume that any optimal solution $(x(\cdot),y(\cdot),u(\cdot))$ of \eqref{OCP}-\eqref{OCP_per} satisfies $y(0)=y(T)=0$ and $x(t)\leq 1$ for every $t\in[0,T]$.
\end{lemma}

\begin{proof}
Let $(\tilde x(\cdot),\tilde y(\cdot),\tilde u(\cdot))$ be an optimal solution of \eqref{OCP}-\eqref{OCP_per}. Since $\tilde x(0)=\tilde x(T)=1$, $x(\cdot)$ takes positive values. By continuity and compactness, let $t_1\in[0,T]$ be such that 
$$
\tilde x(t_1) = \max_{t\in[0,T]}\tilde x(t) .
$$
Extending $(\tilde x(\cdot),\tilde y(\cdot),\tilde u(\cdot))$ by $T$-periodicity and setting 
$$
x^{t_1}(\cdot) = \tilde x(\cdot+t_1), \quad y^{t_1}(\cdot) = \tilde y(\cdot+t_1), \quad u^{t_1}(t) = \tilde u(\cdot+t_1),
$$
(translation in time), the triple $(x^{t_1}(\cdot),y^{t_1}(\cdot),u^{t_1}(\cdot))$ is an optimal solution of \eqref{OCP}-\eqref{OCP_0} (see \ref{A2} in Lemma \ref{lem_prelim}). By \ref{A1} in Lemma \ref{lem_prelim}, we have $x^{t_1}(0)^2+y^{t_1}(0)^2>0$. Now, we set 
$$
\mu=\frac{1}{\sqrt{x^{t_1}(0)^2+y^{t_1}(0)^2}}
$$
and we define 
$$
x(\cdot) = \mu x^{t_1}(\cdot), \quad y(\cdot) = \mu y^{t_1}(\cdot), \quad u(\cdot) = \mu u^{t_1}(\cdot) .
$$
By \ref{A3} in Lemma \ref{lem_prelim}, $(x(\cdot),y(\cdot),u(\cdot))$ is an optimal solution of \eqref{OCP}-\eqref{OCP_0} satisfying, by construction, $x(0)^2+y(0)^2=1$. Moreover, since $x(0)=\mu\tilde x(t_1)$ is the maximum of $x(t)$ over all possible $t\in[0,T]$, and since $x(\cdot)$ is $C^1$ at $t=0$ (by \ref{A4} in Lemma \ref{lem_prelim}), we infer that $\dot x(0)=0$, hence $y(0)=y(T)=0$. The lemma is proved.
\end{proof}

At this step of our analysis, thanks to Lemma \ref{lem_y0}, in what follows we consider the optimal control problem \eqref{OCP} with the state constraint
\begin{equation}\label{x<=1}
x(t)\leq 1\qquad\forall t\in[0,T],
\end{equation}
and with the terminal conditions \eqref{OCP_per_0} (i.e., with respect to \eqref{OCP_per}, we have moreover $y(0)=y(T)=0$).
We have not obtained yet that any optimal solution of that problem satisfies also $x(t)>0$ for every $t\in[0,T]$, i.e., \eqref{0<x<=1}. This will be established in Lemma \ref{lem_xmin>0} in Section \ref{sec_xmin>0}.

By the previous results and in particular by Lemma \ref{lem_per}, given any optimal solution $(x(\cdot),y(\cdot),u(\cdot))$ of \eqref{OCP}-\eqref{OCP_per_0}-\eqref{x<=1} and given any extremal lift $(x(\cdot),y(\cdot), p_x(\cdot),p_y(\cdot), p^0, u(\cdot))$ of it, we have
\begin{align}
& p_x(t)x(t)+p_y(t)y(t) = \mathrm{Cst} = p_x(0)\qquad\forall t\in[0,T], \label{lem1_plusprecis}\\
& p_x(0)=p_x(T), \quad  p_y(0)=p_y(T) .  \label{p_per}
\end{align}
Moreover, recalling that the maximized Hamiltonian $H_1$ is defined by \eqref{def_H1} and is constant along any extremal, we have
\begin{equation}\label{H1geq0}
H_1(x(t),y(t), p_x(t),p_y(t),p^0) = \mathrm{Cst} \geq 0 ,
\end{equation}
and we denote by $H_1$ this constant (which depends on the extremal).
%Indeed, taking $t=0$ and noting that $y(0)=0$, we have $H_1 = \umax\varphi(0)$ if $\varphi(0)>0$, $\umin\varphi(0)$ if $\varphi(0)<0$, and $0$ if $\varphi(0)=0$, hence in all cases $H_1\geq 0$.
Indeed, taking $t=0$ and noting that $y(0)=0$, we have
\begin{equation}\label{H1precis}
H_1 = \max_{\umin\leq v\leq\umax}(v\varphi(0))
= \left\{\begin{array}{ll}
\umax\varphi(0) & \textrm{if}\ \ \varphi(0)>0, \\
\umin\varphi(0) &\textrm{if}\ \ \varphi(0)<0, \\
0 &\textrm{if}\ \ \varphi(0)=0 .
\end{array}\right.
\end{equation}

\subsection{Analysis of the periodic trajectory defined in Theorem \ref{main_thm}}

In this section, we consider the trajectory associated with the control $u$ defined by \eqref{u_opt} and starting at $(1,0)$ (we will prove further that this is the actual optimal trajectory of the optimal control problem \eqref{OCP}-\eqref{OCP_per_0}).

So, we temporarily forget the optimal control problem.
We only consider the control system \eqref{OCP_dyn_x}-\eqref{OCP_dyn_y}-\eqref{OCP_cont_const}, with the initial condition $(x(0),y(0))=(1,0)$.

Let $T>0$ and let $t_1 \in \Big(0, \frac{1}{\omax}\mathrm{Arctan}\big(\frac{\omin}{\omax}\big) \Big)$ be arbitrary, and let $u$ be the control defined by \eqref{u_opt}, i.e., $u(t)=\umax$ if $0<t<t_1$, $u(t)=\umin$ if $t_1<t<T-t_1$ and $u(t)=\umax$ if $T-t_1<t<T$.
This control, which satisfies the control constraint \eqref{OCP_cont_const}, generates a unique trajectory $(x(\cdot),y(\cdot))$ solution of \eqref{OCP_dyn_x}-\eqref{OCP_dyn_y} such that $(x(0),y(0))=(1,0)$.

\begin{proposition}\label{prop_loop}
Given any $T>0$, there exists a unique choice of $t_1 \in \Big(0, \frac{1}{\omax}\mathrm{Arctan}\big(\frac{\omin}{\omax}\big) \Big)$ such that the trajectory $(x(\cdot),y(\cdot))$ defined above is $T$-periodic, i.e., satisfies $x(T)=x(0)=1$ and $y(T)=y(0)=0$.
In turn, this trajectory satisfies the equalities \eqref{Tt1}, \eqref{xopt}, \eqref{cost} stated in Theorem \ref{main_thm}, and the various properties stated in Remark \ref{rem1}.
It is drawn on Figure \ref{fig_thm}.
\end{proposition}

%We stress again that, at this step, we do not know yet that this trajectory is actually the optimal solution of \eqref{OCP}-\eqref{OCP_per}. But the fact that $(x(\cdot),y(\cdot))$ is $T$-periodic shows at least that it is an admissible  trajectory for this optimal control problem.

\begin{proof}
When $t\in(0,t_1)$, we have $u(t)=\umax=\omax^2$ (see \eqref{def_omega}) and integrating \eqref{OCP_dyn_x} with $x(0)=1$ yields $x(t)=\cos(\omax t)$, i.e., the first part of \eqref{xopt}. Along this interval, the curve $(x(\cdot),y(\cdot))$ follows the ellipse of equation $x^2+\frac{y^2}{\omax^2}=1$.
Note that, since $\tan(\omax t_1)<\frac{\omin}{\omax}$, we have
$-\omax\sin(\omax t_1) > -\omin\cos(\omax t_1)$.

\medskip
When $t\in(t_1,T-t_1)$, we have $u(t)=\umin=-\omin^2$ and integrating \eqref{OCP_dyn_x} with $x(t_1)=\cos(\omax t_1)$ yields 
$$
x(t) = A e^{\omin t} + B e^{-\omin t}.
$$
Since $x(t_1)=\cos(\omax t_1)$ and $\dot x(t_1)=-\omax\sin(\omax t_1)$, we infer that
\begin{align*}
A &= \frac{1}{2}\Big( \cos(\omax t_1) - \frac{\omax}{\omin}\sin(\omax t_1) \Big) e^{-\omin t_1}, \\
B &= \frac{1}{2}\Big( \cos(\omax t_1) + \frac{\omax}{\omin}\sin(\omax t_1) \Big) e^{\omin t_1}.
\end{align*}
This gives the second part of \eqref{xopt}. Note that $A>0$. There, the curve $(x(\cdot),y(\cdot))$ follows a hyperbole of equation $x^2-\frac{y^2}{\omin^2} = c(t_1)$ where $c(t_1)=\cos^2(\omax t_1)\Big( 1-\frac{\omax^2}{\omin^2}\tan^2(\omax t_1)\Big)$.

Let us prove that there exists a unique choice of $t_1 \in \Big(0, \frac{1}{\omax}\mathrm{Arctan}\big(\frac{\omin}{\omax}\big) \Big)$ such that this curve crosses the $x$-axis exactly at time $\frac{T}{2}$, i.e., such that $y(\frac{T}{2})=0$.
The latter equality is satisfied if and only if $Ae^{\omin\frac{T}{2}}-Be^{-\omin\frac{T}{2}}=0$, i.e., $e^{\omin T}=\frac{B}{A}$, which leads to the formula \eqref{Tt1} giving $T$ in function of $t_1$. 
Hence, the claim is true if we can prove that the function $t_1\mapsto T(t_1)$ is bijective. This indeed follows from the lemma below.

\begin{lemma}\label{lem_Tt1}
The function $t_1\mapsto T(t_1)$, defined by \eqref{Tt1}, is increasing.
\end{lemma}

\begin{proof}[Proof of Lemma \ref{lem_Tt1}.]
A computation shows that
$$
\frac{dT}{dt_1} = 2\left(1+\frac{\omax^2}{\omin^2}\right) \frac{1}{1-\frac{\omax^2}{\omin^2}\tan^2(\omax t_1)}
$$
and this quantity is positive since $\frac{\omax}{\omin}\tan(\omax t_1)<1$, because $t_1 \in \Big(0, \frac{1}{\omax}\mathrm{Arctan}\big(\frac{\omin}{\omax}\big) \Big)$.
\end{proof}

Therefore, at this step we have proved that there is a unique choice for $t_1$ such that $y(\frac{T}{2})=0$.
Then, completing the construction of the curve on $[\frac{T}{2},T] = [\frac{T}{2},T-t_1] \cup [T-t_1,T]$ gives a trajectory that is symmetric with respect to the $x$-axis. This proves that $x(T)=x(0)=1$ and $y(T)=y(0)=0$.

The cost $\int_0^T u(t)\, dt$ of that trajectory is equal to \eqref{cost}. It is however a nontrivial fact that this cost is always negative. This fact follows from the fact that
\begin{multline*}
\frac{d}{dt_1}\left( -\omin\ln \left( \frac{1+\frac{\omax}{\omin}\tan(\omax t_1)}{1-\frac{\omax}{\omin}\tan(\omax t_1)} \right) + 2 \omax^2 t_1 \right) \\
=
-2 \frac{\omax^2}{\omin^2} (\omin^2+\omax^2) \frac{\tan^2(\omax t_1)}{1-\frac{\omax^2}{\omin^2}\tan^2(\omax t_1)}
\end{multline*}
and this quantity is negative since $\frac{\omax}{\omin}\tan(\omax t_1)<1$, because $t_1 \in \Big(0, \frac{1}{\omax}\mathrm{Arctan}\big(\frac{\omin}{\omax}\big) \Big)$.
Therefore the function that is derivated above is decreasing, and since it is equal to $0$ when $t_1=0$, it is always negative on the open interval $\Big(0, \frac{1}{\omax}\mathrm{Arctan}\big(\frac{\omin}{\omax}\big) \Big)$.

\medskip

The various limits stated in Remark \ref{rem1} are now easily checked.
%To get the equivalent of $x(\frac{T}{2})$, we note that $Ae^{\omin\frac{T}{2}}=Be^{-\omin\frac{T}{2}}$ and thus $x(\frac{T}{2}) = 2Be^{-\omin\frac{T}{2}}$, and we compute the limit of $B$ as $T\rightarrow+\infty$.
To get the equivalent of $x(t)$ on $\big[t_1,\frac{T}{2}\big]$, we note that $A=Be^{-\omin T}$ and thus $x(t) = B ( e^{\omin(t-T)}+e^{-\omin t} )$, and we compute the limit of $B$ as $T\rightarrow+\infty$.
\end{proof}

Again, at this step we do not know yet that the trajectory $(x(\cdot),y(\cdot),u(\cdot))$ defined above is the optimal solution of \eqref{OCP}-\eqref{OCP_per}. But, at least, it is an admissible solution, i.e., it satisfies \eqref{OCP_dyn_x}-\eqref{OCP_dyn_y}-\eqref{OCP_cont_const} and $x(0)=x(T)=1$ and $y(0)=y(T)=0$.

\begin{corollary}\label{cor_loop}
Let $s_1,s_2, T$ be arbitrary real numbers such that $0\leq s_1<s_2\leq T$, and let $\beta>0$ be arbitrary.
There exists a trajectory $(x_{s_1,s_2}(\cdot),y_{s_1,s_2}(\cdot),u_{s_1,s_2}(\cdot))$, solution of \eqref{OCP_dyn_x}-\eqref{OCP_dyn_y}-\eqref{OCP_cont_const} on $[s_1,s_2]$, such that $x_{s_1,s_2}(s_1)=x_{s_1,s_2}(s_2)=\beta$ and $y_{s_1,s_2}(s_1)=y_{s_1,s_2}(s_2)=0$, of negative cost, i.e.,
\begin{equation}\label{cost_s1s2}
\int_{s_1}^{s_2} u_{s_1,s_2}(t)\, dt  <0,
\end{equation}
and such that $x_{s_1,s_2}(t)\leq\beta$ for every $t\in[s_1,s_2]$.
\end{corollary}

\begin{proof}
We apply Proposition \ref{prop_loop} with $T=s_2-s_1$, then we shift in time to get a trajectory $(x_{s_1,s_2}(\cdot),y_{s_1,s_2}(\cdot),u_{s_1,s_2}(\cdot))$ solution of \eqref{OCP_dyn_x}-\eqref{OCP_dyn_y}-\eqref{OCP_cont_const} on $[s_1,s_2]$, with $x_{s_1,s_2}(s_1)=x_{s_1,s_2}(s_2)=1$ and $y_{s_1,s_2}(s_1)=y_{s_1,s_2}(s_2)=0$. Then, using an homothety argument (like in \ref{A3} in Lemma \ref{lem_prelim}), we modify the trajectory (but not the control) so that $x_{s_1,s_2}(s_1)=x_{s_1,s_2}(s_2)=\beta$ and $y_{s_1,s_2}(s_1)=y_{s_1,s_2}(s_2)=0$. The computations done in the proof of Proposition \ref{prop_loop}, in order to prove that the cost \eqref{cost} is negative, show that \eqref{cost_s1s2} holds true.
\end{proof}

Corollary \ref{cor_loop} shows that we can always create a periodic trajectory, with an arbitrarily small period, solution of \eqref{OCP_dyn_x}-\eqref{OCP_dyn_y}-\eqref{OCP_cont_const} and making a loop from $(\beta,0)$ to $(\beta,0)$, for any $\beta>0$, with a negative cost. This nontrivial fact will be useful further. 
%Actually, this fact is reminiscent of the dissipativity property in turnpike theory.

\subsection{Optimal trajectories are contained in $0<x\leq 1$}\label{sec_xmin>0}

\begin{lemma}\label{lem_xmin>0}
Any optimal solution $(x(\cdot),y(\cdot),u(\cdot))$ of \eqref{OCP}-\eqref{OCP_per_0}-\eqref{x<=1} satisfies $0<x(t)\leq 1$ for every $t\in[0,T]$.
Hence, the optimal control problem \eqref{OCP}-\eqref{OCP_per_0}-\eqref{0<x<=1} is equivalent to the optimal control problem \eqref{OCP}-\eqref{OCP_per_0}-\eqref{x<=1}.
\end{lemma}

\begin{proof}
Let $(x(\cdot),y(\cdot),u(\cdot))$ be an optimal solution of \eqref{OCP}-\eqref{OCP_per_0}. We already know that $x(t)\leq 1$ for every $t\in[0,T]$.
By continuity and compactness, we can define
$$
\xmin = \min_{t\in[0,T]}x(t)=x(\tmin)
$$
for some $\tmin\in(0,T)$. Since $\dot x(t)=y(t)$, we must have $y(\tmin)=0$.
We are going to prove that $\xmin>0$.

Let us consider the family of hyperboles 
\begin{equation*}%\label{hyperbole}
\mathcal{H}_c = \left\{ (x,y)\in\R^2\ \mid\ x>0,\ \ x^2-\frac{y^2}{\omin^2} = c \right\}
\end{equation*}
indexed by $c>0$.
Recall that, when $c=c(t_1)=\cos^2(\omax t_1)\Big( 1-\frac{\omax^2}{\omin^2}\tan^2(\omax t_1)\Big)>0$, the hyperbole $\mathcal{H}_c$ contains the trajectory constructed in Proposition \ref{prop_loop}, drawn on Figure \ref{fig_thm}, restricted to $[t_1,T-t_1]$. But now we consider the whole family of hyperboles $(\mathcal{H}_c)_{c>0}$.
We make two useful remarks:
\begin{enumerate}[label=$\bf (H_{\theenumi})$]
\item\label{H1} The hyperboles $\mathcal{H}_c$ are invariant under the dynamics \eqref{OCP_dyn_x}-\eqref{OCP_dyn_y} with $u=\umin$. More precisely, for all $c>0$, $x_0>0$ and $y_0\in\R$ such that $(x_0,y_0)\in\mathcal{H}_c$, the unique solution of the control system \eqref{OCP_dyn_x}-\eqref{OCP_dyn_y}-\eqref{OCP_cont_const} with $u(t)=\umin$ and $x(0)=x_0$, $y(0)=y_0$, satisfies $(x(t),y(t))\in\mathcal{H}_c$ for every $t\in\R$. 
\item\label{H2} Let $y_0<0$ and $y_1>0$ be arbitrary. Given any $T>0$, there exist a unique $c>0$ and a unique $x_0>0$ such that $(x_0,y_0)\in\mathcal{H}_c$ and such that the unique solution of the control system \eqref{OCP_dyn_x}-\eqref{OCP_dyn_y}-\eqref{OCP_cont_const} with $u(t)=\umin$ and $x(0)=x_0$, $y(0)=y_0$, satisfies $y(T)=y_1$. Moreover, $c\rightarrow 0$ if $T\rightarrow +\infty$. 
\end{enumerate}
These preliminary remarks being done, to prove the lemma, we argue by contradiction. Let us assume that $\xmin\leq 0$. 

The curve $(x(\cdot),y(\cdot))$ is periodic, contained in the strip $\xmin\leq x\leq 1$, and satisfies $(x(0),y(0))=(x(T),y(T))=(1,0)$ and $(x(\tmin),y(\tmin))=(\xmin,0)$. 
Moreover, since $\dot x(t)=y(t)$, the curve $t\mapsto x(t)$ is decreasing in the half-plane $y<0$, and increasing in the half-plane $y>0$. It turns clockwise as $t\in[0,T]$. 

\begin{figure}[h]
%\begin{center}
%\resizebox{15cm}{!}{\input fig_hyperboles.pdf_t}
%\end{center}
\centerline{\includegraphics[width=9cm]{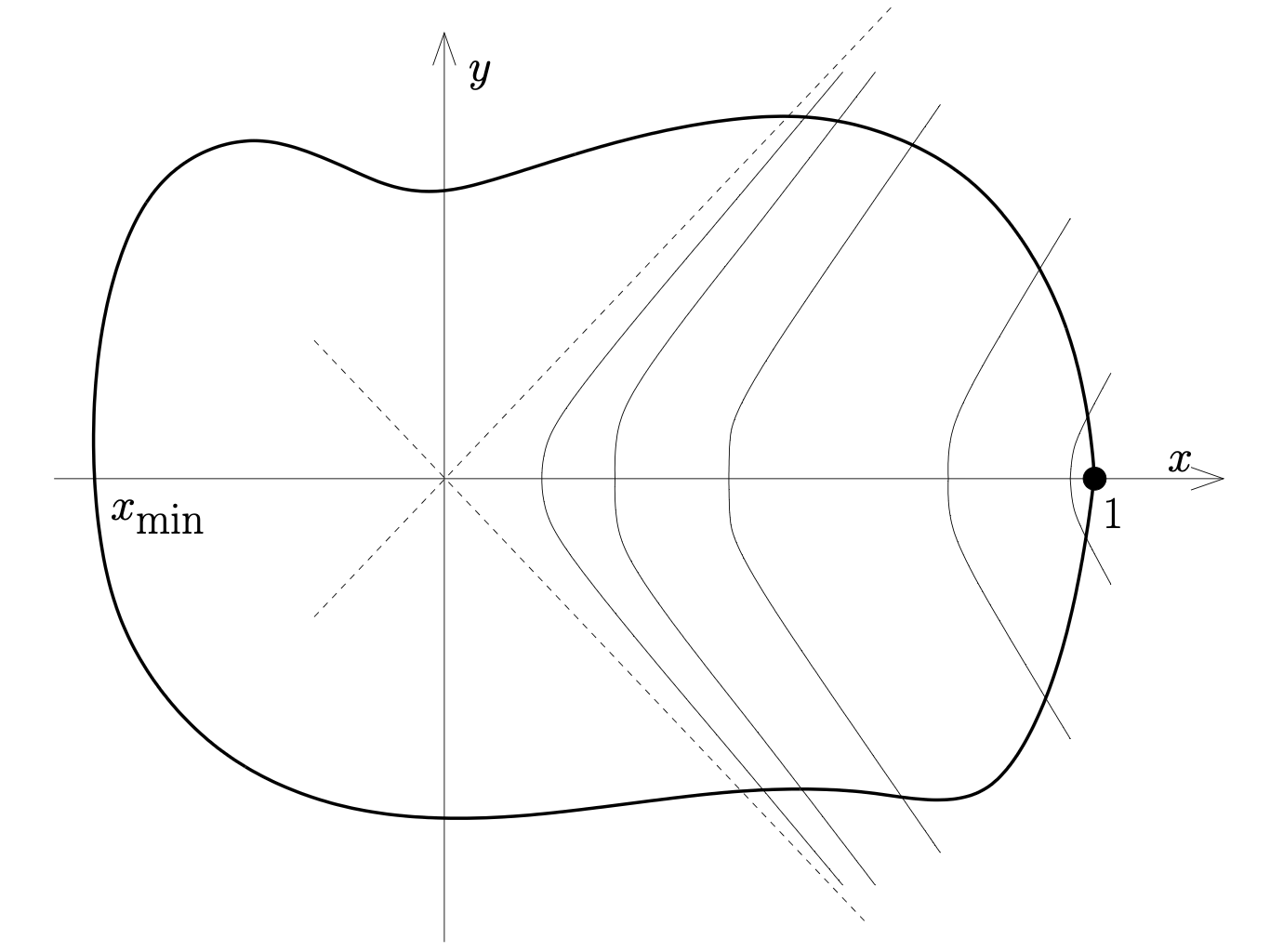}}
\caption{Illustration of the curve intersecting hyperboles.}\label{fig_hyperboles}
\end{figure}

This curve intersects each hyperbole $\mathcal{H}_c$, for every $c\in(0,1]$ (see Figure \ref{fig_hyperboles}).
For every $c\in(0,1]$, choose two points of intersection $(x(s_1),y(s_1))$ and $(x(s_2),y(s_2))$ with $\mathcal{H}_c$ such that $x(s_1)>0$, $y(s_1)<0$ (and $0<s_1<\tmin$) and $x(s_2)>0$, $y(s_2)>0$ (and $\tmin<s_2<T$). 
Then, by Remark \ref{H1}, the trajectory solution of the control system \eqref{OCP_dyn_x}-\eqref{OCP_dyn_y}-\eqref{OCP_cont_const} with $u(t)=\umin$, starting at $(x(s_1),y(s_1))$, follows the hyperbole $\mathcal{H}_c$ and reaches $(x(s_2),y(s_2))$ in a time $\delta>0$ that ranges continuously from $0$ when $c=1$ to $+\infty$ when $c\rightarrow 0$ (by Remark \ref{H2}). Moreover, this hyperbolic arc is obviously optimal for the cost $\int_0^\delta u$, since $u=\umin$ along this arc. 
Therefore, there exists $c\in(0,1]$ such that $\delta=s_2-s_1$. Hence, we have obtained a trajectory consisting of $(x(\cdot),y(\cdot),u(\cdot))$ on $[0,s_1)$, then of the hyperbolic trajectory (with $u=\umin$) on $[s_1,s_2]$ that arrives exactly at $(x(s_2),y(s_2))$ at time $s_2$, then again of $(x(\cdot),y(\cdot),u(\cdot))$ on $(s_2,T]$. 

This new trajectory, obtained by concatenation, has a lower cost than the initial trajectory $(x(\cdot),y(\cdot),u(\cdot))$, since
$$
\int_0^{s_1} u(t)\, dt + (s_2-s_1)\umin + \int_{s_2}^T u(t)\, dt
< 
\int_0^{s_1} u(t)\, dt + \int_{s_1}^{s_2} u(t)\, dt + \int_{s_2}^T u(t)\, dt
$$
unless $u(t)=\umin$ on $[s_1,s_2]$, which is impossible because otherwise the trajectory would follow the hyperbole along this interval and would not penetrate the region $x<0$, while we have assumed that $\xmin<0$.
But then, we have reached a contradiction, because $(x(\cdot),y(\cdot),u(\cdot))$ is optimal.
The lemma is proved.
\end{proof}

\subsection{Optimal controls are bang-bang}\label{sec_nonsing}
As said in Section \ref{sec_PMP}, the maximization condition of the Pontryagin maximum principle has led to \eqref{u_extr}, i.e., $u(t)=\umin$ if $\varphi(t)<0$ and $u(t)=\umax$ if $\varphi(t)>0$, where $\varphi(t) = -p_y(t)x(t)+p^0$ is the switching function defined by \eqref{def_phi}, but we said nothing on the possible value of $u(t)$ on the closed set 
\begin{equation}\label{def_I}
I=\{t\in[0,T]\ \mid\ \varphi(t)=0\}
\end{equation}
where the continuous function $\varphi$ vanishes identically. 
Of course, \eqref{u_extr} is enough if $I$ is of zero Lebesgue measure. However, it could be that $I$ be of positive Lebesgue measure and have a complicated structure.
The next lemma shows that, fortunately, this complicated situation does not occur.
The proof is however not straightforward. 

%As usual in optimal control, we say that we have a \emph{singular arc} if $\varphi$ vanishes identically along a given nontrivial arc (at least, along a given subset of $[0,T]$ of positive Lebesgue measure). We say that $u$ is \emph{bang-bang} if the singular case never occurs, and in that case, $u(t)\in\{\umin,\umax\}$ for almost every $t\in[0,T]$. 

\begin{lemma}\label{lem_nonsing}
Let $(x(\cdot),y(\cdot),u(\cdot))$ be any optimal solution of \eqref{OCP}-\eqref{OCP_per_0}-\eqref{0<x<=1} and let $I$ be defined by \eqref{def_I}.
Then $I$ is of zero Lebesgue measure, and thus the optimal control $u(\cdot)$ is bang-bang, fully described by \eqref{u_extr}.
\end{lemma}

\begin{proof}
In this proof, given any measurable function $f$ on $[0,T]$ and any subset $J\subset[0,T]$ of positive Lebesgue measure, we write ``$f=0$ a.e.~on $J$" to say that $f(t)=0$ for almost every $t\in J$. We recall that, when $f$ is absolutely continuous and thus almost everywhere differentiable, $f=0$ a.e.~on $I$ implies that $\dot f=0$ a.e.~on $I$ (see \cite[Lemma p.~177]{Rudin} or \cite[Theorem 6.3 p.~262 and Remark p.~264]{EvansGariepy}).
Note that, when $g=0$ a.e.~on $I$ for some continuous function $g$ then $g(t)=0$ for every $t\in I$ of positive density in $I$.

\medskip

The proof of the lemma goes by contradiction. Let us assume that the set $I$ defined by \eqref{def_I} is of positive Lebesgue measure. We have $\varphi=0$ on $I$, i.e., 
\begin{equation}\label{pyx}
p_yx=p^0\quad\textrm{on }I.
\end{equation}
Since $I$ is of positive Lebesgue measure, derivating two times \eqref{pyx} and using the dynamical equations \eqref{OCP_dyn_x}, \eqref{OCP_dyn_y}, \eqref{adjoint_x} and \eqref{adjoint_y}, we obtain 
\begin{align}
p_xx&=p_yy \label{pxxpyy}\quad\textrm{a.e.~on }I, \\
-p^0u&=p_xy \label{pxy}\quad\textrm{a.e.~on }I.
\end{align}
There are two cases, depending on whether $p^0=0$ or not.

\medskip

\noindent\textit{First case:} $p^0=0$ (abnormal case). Then $p_xx=p_yy=p_yx=p_xy=0$ a.e.~on $I$.
Note that: (1) we cannot have $x(t)=y(t)=0$ for some $t\in[0,T]$ by \ref{A1} in Lemma \ref{lem_prelim}; (2) we cannot have $p_x(t)=p_y(t)=0$ for some $t\in[0,T]$, for otherwise $(p_x(t),p_y(t),p^0)=(0,0,0)$, which contradicts the nontriviality of this triple stated in the Pontryagin maximum principle. These two remarks and the four above cancellations a.e.~on $I$ lead to a contradiction, and thus this case does not occur.

\medskip

\noindent\textit{Second case:} $p^0\neq 0$ (normal case). Since $(p_x(T),p_y(T),p^0)$ is defined up to scaling, we choose to normalize it, as usual (see \cite{LeeMarkus, Pontryagin, Trelat_bookSB}), so that $p^0=-1$. 
From \eqref{pyx} and \eqref{pxy}, we have then
\begin{align}
& p_yx=-1\quad\textrm{on }I, \label{pyx1} \\
& u=p_xy \quad\textrm{a.e.~on }I \label{pxy1}.
\end{align}
%Indeed, $p_xy$ is constant a.e.~on $I$ because its derivative is $u(p_yy-p_xx)=0$ by \eqref{pxxpyy}.
Recalling that the maximized Hamiltonian $H_1$ defined by \eqref{def_H1} is constant on $[0,T]$ along any extremal and that $H_1\geq 0$ (see \eqref{H1geq0}), and denoting by $H_1$ this constant, 
we have $p_xy=H_1$ on $I$ (because $\varphi=0$ on $I$), and thus, by \eqref{pxy1}, 
\begin{equation}\label{uH1}
u=p_xy=H_1\geq 0\quad\textrm{a.e.~on }I.
\end{equation}
In particular, $u$ is almost everywhere constant on $I$, and this constant is $H_1$ that is nonnegative. We are now going to prove that $H_1=0$.

By \eqref{lem1_plusprecis}, we have $p_x(t)x(t)+p_y(t)y(t)=p_x(0)$ for every $t\in[0,T]$. Using \eqref{pxxpyy}, we infer that 
\begin{equation}\label{pxxpyyC}
p_xx=p_yy=\frac{1}{2}p_x(0) \quad\textrm{a.e.~on }I.
\end{equation}
Let us prove that $p_x(0)=0$. By contradiction, if $p_x(0)\neq 0$, we consider a small interval $[s_1,s_2]\subset[0,T]$, with $s_1<s_2$, such that $[s_1,s_2]\cap I$ has a positive measure, with $s_2-s_1$ small enough so that, by continuity, $x$, $y$, $p_x$ and $p_y$ do not vanish on $[s_1,s_2]$.
By Lemma \ref{lem_xmin>0}, we have $x(t)>0$, and by \eqref{pyx1} we must have $p_y(t)<0$ on $[s_1,s_2]$.
By \eqref{uH1} we must have $\mathrm{sign}(p_x(t))=\mathrm{sign}(y(t))$, and by \eqref{pxxpyyC}, $\mathrm{sign}(p_x(t))=\mathrm{sign}(p_x(0))$ and $\mathrm{sign}(y(t))=-\mathrm{sign}(p_x(0))$, hence $y(t)$ and $p_x(t)$ have opposite signs. But this contradicts \eqref{uH1}. 

Therefore, $p_x(0)=0$. Then, by \eqref{pxxpyyC}, $p_xx=p_yy=0$ a.e.~on $I$, but since $p_y$ and $x$ cannot vanish by \eqref{pyx1}, it follows that $p_x=y=0$ a.e.~on $I$. Derivating $y=0$ a.e.~on $I$, we finally obtain $u=0$ a.e.~on $I$. This also proves that $H_1=0$.

Now, since $H_1=0$ is constant on the interval $[0,T]$, using \eqref{H1precis} (or, taking $t=0$ in \eqref{def_H1} and noting that $y(0)=0$), we must have $\varphi(0)=0$. Since $\varphi(0)=-p_y(0)x(0)-1$ and $x(0)=1$, this implies that $p_y(0)=-1$.

At this step, we have thus obtained that $p_x(0)=0$ and $p_y(0)=-1$.

Inspecting the adjoint differential equations \eqref{adjoint_x}-\eqref{adjoint_y}, we observe that $(-p_y(\cdot),p_x(\cdot),u(\cdot))$ is solution of \eqref{OCP_dyn_x}-\eqref{OCP_dyn_y}, like the triple $(x(\cdot),y(\cdot),u(\cdot))$, with the same control $u(\cdot)$ and the same initial condition $(1,0)$ (because $(-p_y(0),p_x(0))=(1,0)$). By Cauchy uniqueness it follows that $p_x(t) = y(t)$ and $p_y(t)=-x(t)$ for every $t\in[0,T]$. But then, since $H_1=0$ is constant on $[0,T]$, using again \eqref{def_H1} we infer that
$$
y(t)^2 + \max_{\umin\leq v\leq\umax}(v\varphi(t)) = 0 \qquad\forall t\in[0,T] ,
$$
and since the above maximum is nonnegative, we must have $y(t)^2=0$ thus $y(t)=0$ for every $t\in[0,T]$. Using \eqref{OCP_dyn_x}-\eqref{OCP_dyn_y} and $x(0)=1$, this implies that $(x(t),y(t))=(1,0)$ for every $t\in[0,T]$ and $u(t)=0$ for almost every $t\in[0,T]$.
But this trivial solution is not optimal because its cost is equal to $0$, while the loop trajectory constructed in Proposition \ref{prop_loop} has a negative cost (in time $T$), thus does better.
We have thus obtained a contradiction, and the lemma is proved.
%Let us now find a contradiction. Since $\varphi=0$ a.e.~on $I$, we have $x^2=1$, and since $0<x(t)\leq 1$,  we get $x(t)=1$ for every $t\in I$.
%Derivating in time, we must have $y=0$ a.e.~on $I$. 
%Then $H=y^2+u(x^2-1)$ and since $x\leq 1$ $u$ ne peut jamais prendre la valeur $\umax$, d'o\`u la contradiction.
\end{proof}

\subsection{Uniqueness of the optimal trajectory}\label{sec_uniqueness}

\begin{proposition}\label{prop_unique}
The curve constructed in Proposition \ref{prop_loop} is the unique optimal solution of \eqref{OCP}-\eqref{OCP_per_0}-\eqref{0<x<=1}. 
\end{proposition}

\begin{proof}
Let $(x(\cdot),y(\cdot),u(\cdot))$ be an optimal solution of \eqref{OCP}-\eqref{OCP_per_0}-\eqref{0<x<=1}. By the proof of Lemma \ref{lem_xmin>0}, we have $0<\xmin\leq x(t)\leq 1$ for every $t\in[0,T]$, where $\xmin>0$  is the minimal value of $x(t)$.
Besides, as a consequence of Lemma \ref{lem_nonsing}, the time interval $[0,T]$ is a countable union of open intervals along which either $u(t)=\umin$, and then the curve $(x(\cdot),y(\cdot))$ follows clockwise an arc of hyperbole (according to the computations done in the proof of Proposition \ref{prop_loop}), or $u(t)=\umax$, and then the curve $(x(\cdot),y(\cdot))$ follows clockwise an arc of ellipse.
Moreover, since $\dot x=y$ and $\dot y=-ux$, the function $t\mapsto x(t)$ is decreasing in the half-plane $y<0$, and increasing in the half-plane $y>0$. The curve $(x(\cdot),y(\cdot))$ turns clockwise and can cross the $x$-axis only with a vertical tangent. 

The structure of the control $u(\cdot)$ may be complicated, though: a priori, it may switch an infinite %(not necessarily countable) 
number of times, but the set $\{t\in[0,T]\ \mid\ \varphi(t)\}$ is of measure zero. 

Like in Lemma \ref{lem_xmin>0}, we denote by $\xmin$ the minimal value of $x(t)$ over all possible $t\in[0,T]$, and let $\tmin\in(0,T)$ be such that $x(\tmin)=\xmin$.
The function $t\mapsto x(t)$ may fail to be decreasing on $[0,\tmin]$: it may happen that the curve $(x(\cdot),y(\cdot))$ crosses the $x$-axis at some time $\bar t\in(0,\tmin)$, i.e., with $\xmin<x(\bar t)<1$ and $y(\bar t)=0$, then penetrates in the region $y>0$ (where $t\mapsto x(t)$ is increasing) and comes back later in the region $y<0$ (see Figure \ref{fig_surgery} on the left).

\smallskip
The proof goes by contradiction. Assuming that $(x(\cdot),y(\cdot),u(\cdot))$ differs from the trajectory constructed in Proposition \ref{prop_loop}, we are going to build a new trajectory with the same terminal conditions, having a lower cost, thus reaching a contradiction.
Without loss of generality, we assume that the trajectories differ in the half-plane $y\leq 0$.

Let us write the interval $[0,\tmin]$ as the countable union of some intervals $I_k$ and $J_p$ (all disjoint two by two), for $k$ and $p$ ranging in some countable set, such that:
%curve $(x(\cdot),y(\cdot))$ consists of a countable union of arcs with the following properties: 
\begin{itemize}
\item on each open interval $I_k$, $x$ is decreasing and $y\leq 0$, %either $x$ is decreasing and $y<0$, or $x$ is increasing and $y>0$,
\item on each closed interval $J_p$, the curve $(x(\cdot),y(\cdot))$ is periodic, i.e., $x(\min J_p)=x(\max J_p)$ and $y(\min J_p)=y(\max J_p)\leq 0$ (it makes one or several loops),
\end{itemize}
as illustrated on Figure \ref{fig_surgery}, on the left. 

\begin{figure}[h]
%\begin{center}
%\resizebox{15cm}{!}{\input fig_surgery.pdf_t}
%\end{center}
\centerline{\includegraphics[width=14cm]{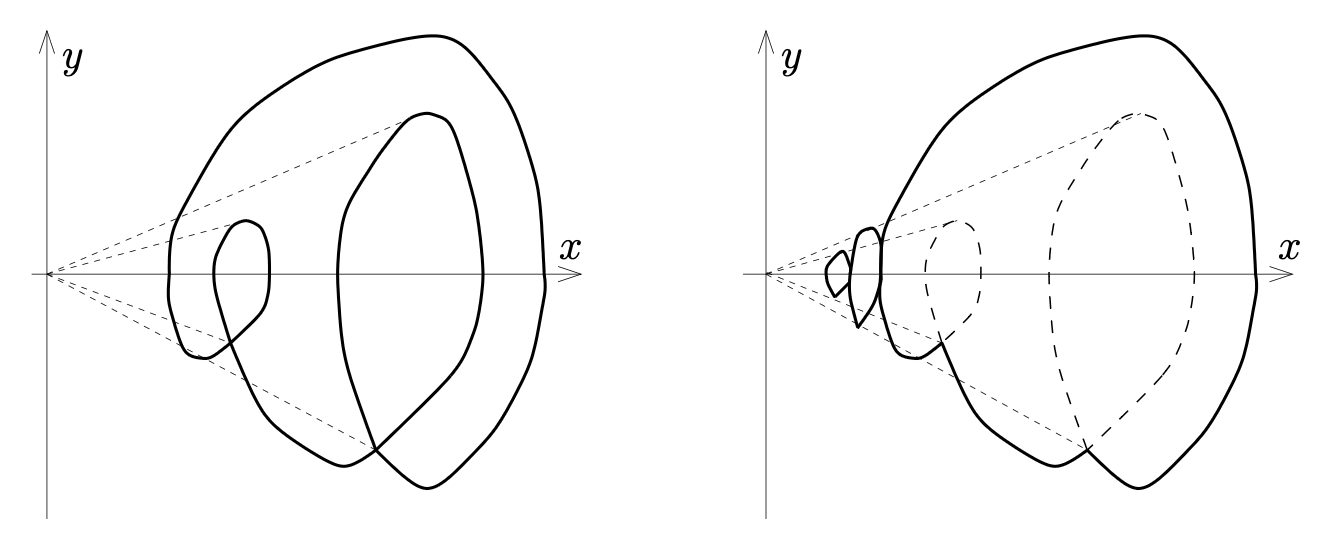}}
\caption{Surgery. On the left figure, the initial curve (in plain). On the right figure, periodic arcs have been cut and homothetized to the left.}\label{fig_surgery}
\end{figure}

The proof now goes in three steps.

\medskip

\noindent{\it Step 1.}
We first make some ``surgery", in order to build another curve having the same cost (and thus, being optimal as well): take an arbitrary periodic arc, parametrized on an interval $J_p$, cut it from the original curve, apply to it an homothety (from the origin) so as to glue it to the left of the curve; then repeat this operation for all such periodic arcs index by $p$ (this can be done in any order, and there may be an infinite countable number of such arcs).
By \ref{A3} in Lemma \ref{lem_prelim}, this homothety does not affect the control along $J_p$ and thus does not change its cost.

Now, we reparametrize the newly obtained curve (of which the minimal value $\tilde x_{\mathrm{min}}$ of the $x$-component is now lower than $\xmin$), by shifting in time and concatenating, so that the new trajectory, denoted $(\tilde x(\cdot),\tilde y(\cdot),\tilde u(\cdot))$, is still defined on $[0,T]$ and now consists of a countable number of successive arcs as follows: there exists $\tilde t_{\mathrm{min}}\in(0,\tmin]$ such that, on $[0,\tilde t_{\mathrm{min}}]$, the function $t\mapsto\tilde x(t)$ is decreasing, with $\tilde x(\tilde t_{\mathrm{min}}) = \tilde x_{\mathrm{min}}$ that is the minimal value of $\tilde x(t)$ over all possible $t\in[0,T]$.
The part of $(\tilde x(\cdot),\tilde y(\cdot))$ that is contained in the half-plane $y\leq 0$ is drawn on Figure \ref{fig_surgery_half}. 
\begin{figure}[h]
%\begin{center}
%\resizebox{11cm}{!}{\input fig_surgery_half.pdf_t}
%\end{center}
\centerline{\includegraphics[width=7cm]{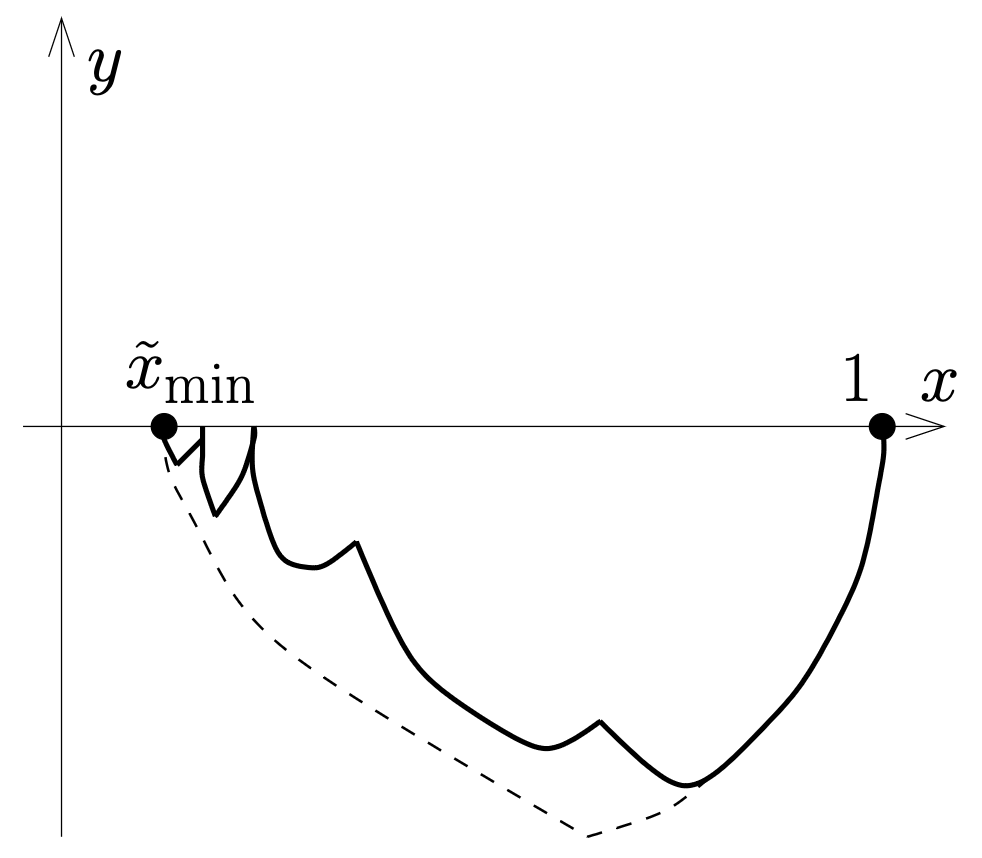}}
\caption{In plain, the curve $(\tilde x(\cdot),\tilde y(\cdot))$ restricted to $[0,\tilde t_{\mathrm{min}}]$, obtained after surgery. In dashed, the trajectory considered in Step 2.}\label{fig_surgery_half}
\end{figure}
On the interval $[\tilde t_{\mathrm{min}},\tmin]$, we keep the pieces of the homothetized arcs contained in $y\geq 0$ until we reach the point $(\xmin,0)$, and then we set $(\tilde x(t),\tilde y(t),\tilde u(t))=(x(t),y(t),u(t))$ on the time interval $[\tmin,T]$, this piece steering $(\xmin,0)$ to $(1,0)$ (see Figure \ref{fig_surgery} on the right, in the half-plane $y\geq 0$).
The cost of the new trajectory $(\tilde x(\cdot),\tilde y(\cdot),\tilde u(\cdot))$ is the same as the initial one $(x(\cdot),y(\cdot),u(\cdot))$, i.e., 
\begin{equation}\label{cost_utildeu}
\int_0^T \tilde u(t)\, dt = \int_0^T u(t)\, dt.
\end{equation}
In particular, $(\tilde x(\cdot),\tilde y(\cdot),\tilde u(\cdot))$ is an optimal solution of \eqref{OCP}-\eqref{x<=1}-\eqref{OCP_per_0}, like $(x(\cdot),y(\cdot),u(\cdot))$.
Actually, the control $\tilde u(\cdot)$ is a kind of rearrangement of the initial control $u(\cdot)$, obtained by shifting in time some subintervals and rearranging them differently, without changing the time interval $[0,T]$.
Moreover, $(\dot{\tilde x}(t),\dot{\tilde y}(t))\neq (0,0)$ for almost every $t\in[0,T]$, and even, for every $t\in[0,T]$ except maybe on a countable subset of $[0,T]$ thanks to the monotonicity of $\tilde x(\cdot)$. This shows that $(\tilde x(\cdot),\tilde y(\cdot))$ is a Jordan curve, what we are going to use in the second step.

Note also that, since we have assumed (by contradiction) that $(x(\cdot),y(\cdot),u(\cdot))$ differs from the trajectory constructed in Proposition \ref{prop_loop}, the same holds true for the new trajectory $(\tilde x(\cdot),\tilde y(\cdot),\tilde u(\cdot))$.

\medskip

\noindent{\it Step 2.}
As a second step, based on $(\tilde x(\cdot),\tilde y(\cdot),\tilde u(\cdot))$, we are now going to construct a new trajectory $(\hat x(\cdot),\hat y(\cdot),\hat u(\cdot))$ defined on $[0,T]$, with a (strictly) lower cost, thus reaching a contradiction. 

The strategy is completely different from the one used in Step 1. To explain it, we start by noting that the control system \eqref{OCP_dyn_x}-\eqref{OCP_dyn_y} is control-affine, written as
\begin{equation}\label{cont_syst_affine}
\dot q(t) = X(q(t))+u(t)Y(q(t))
\end{equation}
where $q=(x,y)$ and the two vector fields $X$ and $Y$ on $\R^2$ are defined by
$$
X(q) = y\frac{\partial}{\partial x} \qquad \textrm{and}\qquad Y(q) = -x\frac{\partial}{\partial y} .
$$
Let $\alpha$ and $\beta$ be the smooth differential one-forms on $(0,+\infty)\times(-\infty,0)$ defined by $\alpha = \frac{dx}{y}$ and $\beta = -\frac{dy}{x}$.
Note that $d\alpha = \frac{dx\wedge dy}{y^2}$ and $d\beta = \frac{dx\wedge dy}{x^2}$.
By definition of $\alpha$ and $\beta$, we have $\alpha(X)=1$ and $\alpha(Y)=0$, $\beta(X)=0$ and $\beta(Y)=1$. Hence $\alpha_{q(t)}(\dot q(t))=1$ and $\beta_{q(t)}(\dot q(t))=u(t)$ for almost every $t$, for any given solution $q(\cdot)$ of \eqref{cont_syst_affine} such that $y(t)\neq 0$ almost everywhere. %To obtain the latter fact for $\alpha$, since $\alpha$ is not well defined at $y=0$, we use a limit argument and the fact that $(\dot{\tilde x}(t),\dot{\tilde y}(t))\neq (0,0)$ almost everywhere. 

The use of one-differential forms, combined with the application of the Green-Riemann (Stokes) theorem, to get optimality properties for trajectories in the plane is known in optimal control (see \cite{HermesLaSalle}, see also \cite{BonnardChyba} where the one-form $\alpha$ is referred to as the clock form), but the above form $\beta$ is nonclassical, up to our knowledge.

The reasoning goes as follows. 
First of all, we consider, in the half-plane $y\leq 0$, the trajectory constructed in Proposition \ref{prop_loop}, denoted $(\hat x(\cdot),\hat y(\cdot),\hat u(\cdot))$, that steers the control system \eqref{OCP_dyn_x}-\eqref{OCP_dyn_y}-\eqref{OCP_cont_const} from $(1,0)$ to $(\tilde x_{\mathrm{min}},0)$ in time denoted by $\hat t_{\mathrm{min}}>0$ (and thus, it goes from $(1,0)$ to $(1,0)$ in time $2\hat t_{\mathrm{min}}$).
The curve $(\hat x(\cdot),\hat y(\cdot))$ is drawn in dashed on Figure \ref{fig_surgery_half}. 

Hence, we now have two trajectories steering the control system \eqref{OCP_dyn_x}-\eqref{OCP_dyn_y}-\eqref{OCP_cont_const} from $(1,0)$ to $(\tilde x_{\mathrm{min}},0)$:
\begin{itemize}
\item $(\tilde x(\cdot),\tilde y(\cdot),\tilde u(\cdot))$ (in plain on Figure \ref{fig_surgery_half}) does it in time $\tilde t_{\mathrm{min}}$,
\item $(\hat x(\cdot),\hat y(\cdot),\hat u(\cdot))$ (in dashed on Figure \ref{fig_surgery_half}) does it in time $\hat t_{\mathrm{min}}$.
\end{itemize}
Let us consider the curve $\Gamma$, making a counterclockwise loop from $(1,0)$ to $(1,0)$, by first following the curve $(\tilde x(\cdot),\tilde y(\cdot))$, going from $(1,0)$ to $(\tilde x_{\mathrm{min}},0)$ in time $\tilde t_{\mathrm{min}}$, and then following backward in time the curve $(\hat x(\cdot),\hat y(\cdot))$, going from $(\tilde x_{\mathrm{min}},0)$ to $(1,0)$ in time $\hat t_{\mathrm{min}}$.
Since $(\dot{\tilde x}(t),\dot{\tilde y}(t))\neq (0,0)$ for almost every $t\in[0,\tilde t_{\mathrm{min}}]$ and $(\dot{\hat x}(t),\dot{\hat y}(t))\neq (0,0)$ for almost every $t\in[0,\hat t_{\mathrm{min}}]$, it follows that $\Gamma$ is a Jordan curve.
%Since $\Gamma$ is a limit of Jordan curves (consisting of cutting the singularities at $y=0$ and of slightly spreading, if necessary, the two curves at places where they intersect), at the limit, 
Applying the Green-Riemann theorem yields
$$
\int_\Gamma\alpha = \int_\Omega d\alpha > 0
\qquad\textrm{and}\qquad
\int_\Gamma\beta = \int_\Omega d\beta > 0
$$
for $i=1,2$, because we have $d\omega_i>0$ in both cases (one can note that $\int_\Omega d\alpha$ converges in spite of the singularity $1/y^2$ because we integrate on a cusp region).
This gives 
\begin{equation}\label{cost_uhat}
\tilde t_{\mathrm{min}}>\hat t_{\mathrm{min}}
\qquad\textrm{and}\qquad
\int_0^{\tilde t_{\mathrm{min}}} \tilde u(t)\, dt > \int_0^{\hat t_{\mathrm{min}}} \hat u(t)\, dt.
\end{equation}

\medskip

\noindent{\it Step 3.}
Now comes the final construction, which raises a contradiction.
Recall that, by the above construction, $0 < \hat t_{\mathrm{min}} < \tilde t_{\mathrm{min}} \leq \tmin < T$, $\tilde x_{\mathrm{min}} = \hat x(\hat t_{\mathrm{min}}) = \tilde x(\tilde t_{\mathrm{min}})$ and $\xmin = \tilde x(\tmin) =  x(\tmin)$.
In particular, we set $\Delta t = \tilde t_{\mathrm{min}} - \hat t_{\mathrm{min}} >0$.
We consider the curve $(x_1(\cdot),y_1(\cdot))$ (of control $u_1(\cdot)$) that is defined as the concatenation of three pieces:
\begin{enumerate}
\item First, follow the curve $(\hat x(\cdot),\hat y(\cdot))$ from $(1,0)$ to $(\tilde x_{\mathrm{min}},0)$ (contained in $y\leq 0$). This curve is defined on the time interval $[0,\hat t_{\mathrm{min}}]$.
\item Then, follow the curve $(\tilde x(\cdot),\tilde y(\cdot))$ from $(\tilde x_{\mathrm{min}},0)$ to $(\xmin,0)$ (contained in $\tilde x_{\mathrm{min}}\leq x\leq\xmin$, $y\geq 0$). This curve is defined on the time interval $[\tilde t_{\mathrm{min}},\tmin]$, so we need to shift it in time, by advancing it by $\Delta t = \tilde t_{\mathrm{min}} - \hat t_{\mathrm{min}} >0$.
\item Finally, follow the curve $(\tilde x(\cdot),\tilde y(\cdot))=(x(\cdot),y(\cdot))$ from $(\xmin,0)$ to $(1,0)$ (contained in $y\geq 0$). This curve is defined on the time interval $[\tmin,T]$, so we need to shift it in time, by advancing it by $\Delta t = \tilde t_{\mathrm{min}} - \hat t_{\mathrm{min}} >0$.
\end{enumerate}
Precisely, the control $u_1(\cdot)$ is defined by
$$
u_1(t) = \left\{ \begin{array}{ll}
\hat u(t) & \textrm{if}\ \  0 < t < \hat t_{\mathrm{min}} , \\ 
\tilde u(t+\Delta t) & \textrm{if}\ \  \tilde t_{\mathrm{min}}-\Delta t < t < \tmin  - \Delta t , \\ 
\tilde u(t+\Delta t) = u(t+\Delta t) & \textrm{if}\ \  \tmin  - \Delta t < t < T-\Delta t .
\end{array}\right.
$$
The trajectory $(x_1(\cdot),y_1(\cdot),u_1(\cdot))$ is a solution of \eqref{OCP_dyn_x}-\eqref{OCP_dyn_y}-\eqref{OCP_cont_const}, satisfies $x_1(\cdot)\leq 1$ and steers the control system from $(1,0)$ to $(1,0)$ in time $T-\Delta t$. Using the strict inequality \eqref{cost_uhat} and the equality \eqref{cost_utildeu}, its cost is
\begin{equation*}
\begin{split}
\int_0^{T-\Delta t} u_1(t)\, dt 
&= \int_0^{\hat t_{\mathrm{min}}} \hat u(t)\, dt + \int_{\tilde t_{\mathrm{min}}}^{\tmin} \tilde u(t)\, dt + \int_{\tmin}^T \tilde u(t)\, dt \\
&< \int_0^{\tilde t_{\mathrm{min}}} \tilde u(t)\, dt + \int_{\tilde t_{\mathrm{min}}}^{\tmin} \tilde u(t)\, dt + \int_{\tmin}^T \tilde u(t)\, dt 
= \int_0^T \tilde u(t)\, dt = \int_0^T u(t)\, dt ,
\end{split}
\end{equation*}
i.e., its cost is lower than the cost of the initial trajectory $(x(\cdot),y(\cdot),u(\cdot))$.

To get a trajectory defined on the same interval of time $[0,T]$, we finally extend the trajectory $(x_1(\cdot),y_1(\cdot),u_1(\cdot))$ to the interval $[0,T]$, by concatenating it with a small loop of negative cost on $[T-\Delta t,T]$. 
This can be done because, by Corollary \ref{cor_loop}, there exists a trajectory $(x_{\Delta t}(\cdot),y_{\Delta t}(\cdot),u_{\Delta t}(\cdot))$, solution of \eqref{OCP_dyn_x}-\eqref{OCP_dyn_y}-\eqref{OCP_cont_const} on $[T-\Delta t,T]$, such that $x_{\Delta t}(T-\Delta t)=x_{\Delta t}(T)=1$ and $y_{\Delta t}(T-\Delta t)=y_{\Delta t}(T)=0$, satisfying $x_{\Delta t}(t)\leq 1$ for every $t\in[T-\Delta t,T]$, and of negative cost $\int_{T-\Delta t}^T u_{\Delta t}(t)\, dt  <0$. 

Then, concatenating $(x_1(\cdot),y_1(\cdot),u_1(\cdot))$ on $[0,T-\delta T]$ with $(x_{\Delta t}(\cdot),y_{\Delta t}(\cdot),u_{\Delta t}(\cdot))$ on $[T-\Delta t,T]$, we obtain a solution of \eqref{OCP_dyn_x}-\eqref{OCP_dyn_y}-\eqref{OCP_cont_const} on $[0,T]$, lying in $0<x\leq 1$, steering the control system from $(1,0)$ to $(1,0)$ in time $T$, of cost (strictly) lower than $\int_0^T u(t)\, dt$. This raises a contradiction with the optimality of the trajectory $(x(\cdot),y(\cdot),u(\cdot))$.

The proposition is proved.
\end{proof}

\subsection{Conclusion: end of the proof of Theorem \ref{main_thm}}
In Section \ref{sec_firstreduction}, we have first reduced the initial optimal control problem \eqref{OCP}-\eqref{OCP_0} to the optimal control problem \eqref{OCP}-\eqref{OCP_per}, noting that any other solution of \eqref{OCP}-\eqref{OCP_0} can be deduced from the one of \eqref{OCP}-\eqref{OCP_per} by simple geometric considerations. Then, in Section \ref{sec_secondreduction}, we have shown that the optimal control problem \eqref{OCP}-\eqref{OCP_per} is equivalent to the optimal control problem \eqref{OCP}-\eqref{OCP_per_0}-\eqref{x<=1}, which itself is equivalent, by Lemma \ref{lem_xmin>0} in Section \ref{sec_xmin>0}, to the optimal control problem \eqref{OCP}-\eqref{OCP_per_0}-\eqref{0<x<=1}.
Finally, we have proved in Proposition \ref{prop_unique} (in Section \ref{sec_uniqueness}) that the optimal control problem \eqref{OCP}-\eqref{OCP_per_0}-\eqref{0<x<=1} has a unique solution, which is the one constructed in Proposition \ref{prop_loop}. This proves Theorem \ref{main_thm}.

\section{Additional remarks}\label{sec_add_rem}

\subsection{Uniqueness of the adjoint}

\begin{lemma}\label{lem_adjoint_unique}
The unique optimal solution $(x(\cdot),y(\cdot),u(\cdot))$ of \eqref{OCP}-\eqref{OCP_0}, described in Theorem \ref{main_thm}, has a unique extremal lift $(x(\cdot),y(\cdot),p_x(\cdot),p_y(\cdot),p^0,u(\cdot))$ up to scaling on the adjoint triple, which is normal. Moreover, normalizing it so that $p^0=-1$, and the adjoint vector is given by
$$
(p_x(t), p_y(t)) = p_y(0)(-y(t), x(t)) \qquad \forall t\in[0,T] ,
$$
and we have $p_y(0)<0$.
\end{lemma}

\begin{proof}
As in the proof of Lemma \ref{lem_nonsing} in Section \ref{sec_nonsing}, by the adjoint differential equations \eqref{adjoint_x}-\eqref{adjoint_y}, we observe that $(-p_y(\cdot),p_x(\cdot),u(\cdot))$ is solution of \eqref{OCP_dyn_x}-\eqref{OCP_dyn_y}, like the triple $(x(\cdot),y(\cdot),u(\cdot))$, with the same control $u(\cdot)$ (thus, with the same cost) but not with the same initial condition. However, the pair $(-p_y(\cdot),p_x(\cdot))$ is $T$-periodic. By uniqueness of the solution of the optimal control problem \eqref{OCP}-\eqref{OCP_per}, and using Lemma \ref{lem_prelim}, we infer that $(-p_y(\cdot),p_x(\cdot))$ is equal to $(x(\cdot),y(\cdot))$ up to scaling and to shifting in time. But since both curves are generated by the same control $u(\cdot)$ and thus consist of following two arcs, an arc of ellipse and an arc of hyperbole (see Figure \ref{fig_thm}, it follows that the shift in time is zero and thus there exists $\lambda\neq 0$ such that $x(t)=-\lambda p_y(t)$ and $y(t)=\lambda p_x(t)$ for every $t\in[0,T]$. In particular, we must have $p_x(0)=0$ and, using for instance the positive sign of $H_1$, necessarily $\lambda>0$ and $p_y(0)=-\frac{1}{\lambda}<0$. 
The switching function is then $\varphi(t)=\frac{x(t)^2}{\lambda}+p^0$. Therefore $p^0<0$ (otherwise, if $p^0=0$ then $\varphi>0$ and there would be no switching) and we can normalize the adjoint so that $p^0=-1$. 
\end{proof}

\subsection{Shooting method}

Setting $\lambda=-1/p_y(0)>0$ as in the proof of Lemma \ref{lem_adjoint_unique}, we have $\lambda\varphi = x^2-\lambda$, and we thus infer from \eqref{u_extr} that
\begin{equation}\label{u_extr_lambda}
u(t) = \left\{ \begin{array}{ll}
\umin & \textrm{if}\ \ x(t)^2<\lambda, \\[1mm]
\umax & \textrm{if}\ \ x(t)^2>\lambda.
\end{array}\right.
\end{equation}
By the way, since $x(0)=1$, we must have, actually, $0<\lambda<1$. 

The expression \eqref{u_extr_lambda} for the optimal control gives an alternative way to compute optimal trajectories by implementing the \emph{shooting method}. Here, there is only one shooting parameter, that is $\lambda\in(0,1)$ (and actually, $\lambda=-1/p_y(0)$ as said above), and this parameter must be tuned so that, when one integrates the differential equations \eqref{OCP_dyn_x}-\eqref{OCP_dyn_y}, with initial condition $(x(0),y(0))=(1,0)$, one must have $x(T)=1$ at time $T$. 

This is the method that has been implemented in \cite{LazarusTrelat_PNAS}.

\subsection{Turnpike phenomenon}\label{sec_turnpike}
As said in Remark \ref{rem1}, when $T\rightarrow+\infty$ the arc of the curve $t\mapsto(x(t),y(t))$ on $\big[t_1,\frac{T}{2}\big]$, along which $u=\umin$, converges to the segment joining the point $\Big( \sqrt{\frac{\umax}{\umax-\umin}}, -\sqrt{\frac{-\umin\umax}{\umax-\umin}} \Big)$ to the point $(0,0)$. This is because the arc of hyperbole converges to a ``V-shape" (with the ``V" rotated by $-\pi/2$), as it can be seen on the numerical simulation reported on Figure \ref{fig_Vshape}.

\begin{figure}[h]
\centerline{\includegraphics[width=9cm]{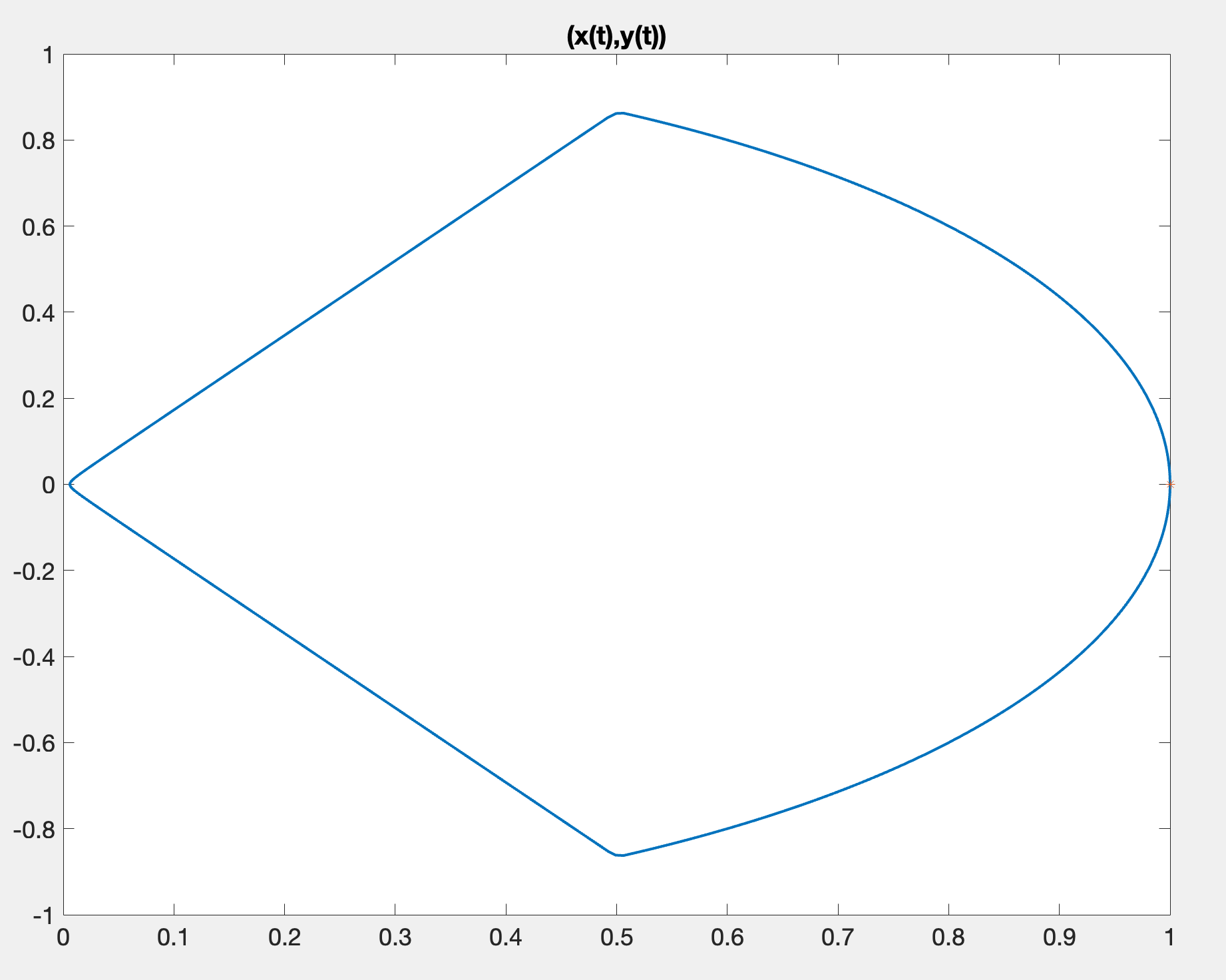}}
\caption{Numerical simulation when $T$ is large. Here: $\umin=-3$, $\umax=1$, $T=8$.}\label{fig_Vshape}
\end{figure}

The turnpike phenomenon refers to the following typical property enjoyed by optimal trajectories of a large classes of optimal control problems in large horizon of time: when $T$ is large, the optimal trajectories (as well as their control and adjoint vector) tend to spend most of their time near a steady-state, itself being the optimal solution of a static optimization problem (see \cite{TrelatZuazua_JDE2015}). 

This phenomenon is evident here: the optimal steady-state is $(\bar x,\bar y,\bar u)=(0,0,\umin)$. %, as already alluded to in Section \ref{sec_OCP}; it is the optimal solution when one relaxes the nontriviality constraint \eqref{OCP_>0}.
The turnpike phenomenon observed on Figure \ref{fig_Vshape} is particularly intuitive for the optimal control problem \eqref{OCP}-\eqref{OCP_0}: as said in Section \ref{sec_OCP}, when we relax the constraint \eqref{OCP_>0} then the optimal solution is the steady-state $(\bar x,\bar y,\bar u)=(0,0,\umin)$. Under the nontriviality constraint \eqref{OCP_>0}, optimal trajectories tend to ``reproduce" as much as possible this optimal steady-state. This is the turnpike phenomenon.

In some sense, this observation ``trivializes" the computation of optimal trajectories in large time: when $T\rightarrow+\infty$, the unique optimal trajectory passing through the point $(1,0)$ consists of two pieces, one being an arc of ellipse with a shape of a ``U" rotated by $\pi/2$, and the other being the union of two segments (``V-shape"), as said at the beginning of this section. Said in other words, in large time the optimal trajectories look like ``teardrops".

\subsection{Locally optimal butterfly solutions}
In Theorem \ref{main_thm} we have proved that any \emph{globally} optimal solution of \eqref{OCP}-\eqref{OCP_0} is entirely contained in the half-plane $x>0$ or $x<0$. 

In this section, we present a family of \emph{locally} but not \emph{globally} optimal solutions of \eqref{OCP}-\eqref{OCP_0}, that we call ``butterfly solutions" (because of their shape). Such solutions cross the $y$-axis. 

So, let us now consider the optimal control problem \eqref{OCP}-\eqref{OCP_0}, in which we add the constraint
\begin{equation}\label{OCP_cross}
\exists t\in[0,T]\ \mid\ x(t)=0.
\end{equation}
Using Lemma \ref{lem1} and in particular homotheties and shifting in time, the optimal control problem \eqref{OCP}-\eqref{OCP_0}-\eqref{OCP_cross} can be reduced to the optimal control problem \eqref{OCP}-\eqref{OCP_per_0}-\eqref{OCP_cross}, i.e., with $x(0)=x(T)=1$, $y(0)=y(T)=0$, under \eqref{OCP_cross} and under the state constraint
\begin{equation}\label{-1<=x<=1}
-1\leq x(t)\leq 1\qquad\forall t\in[0,T].
\end{equation}
With similar arguments as the ones developed to prove Theorem \ref{main_thm}, it can be proved that, given any $T\geq 2\pi$, the optimal control problem \eqref{OCP}-\eqref{OCP_per_0}-\eqref{OCP_cross}-\eqref{-1<=x<=1} has a unique solution $(x_b(\cdot),y_b(\cdot),u_b(\cdot))$ (the index ``b" refers to ``butterfly"), which is symmetric with respect to the $x$-axis and to the $y$-axis, and which can be explicitly described similarly to what was done in Theorem \ref{main_thm}.
The optimal curve consists of two arcs of an ellipse that are symmetric with respect to the $y$-axis, in-between of which there are two arcs of hyperbole that are symmetric with respect to the $x$-axis (see Figure \ref{fig_butterfly}).

\begin{figure}[h]
\centerline{\includegraphics[width=13cm]{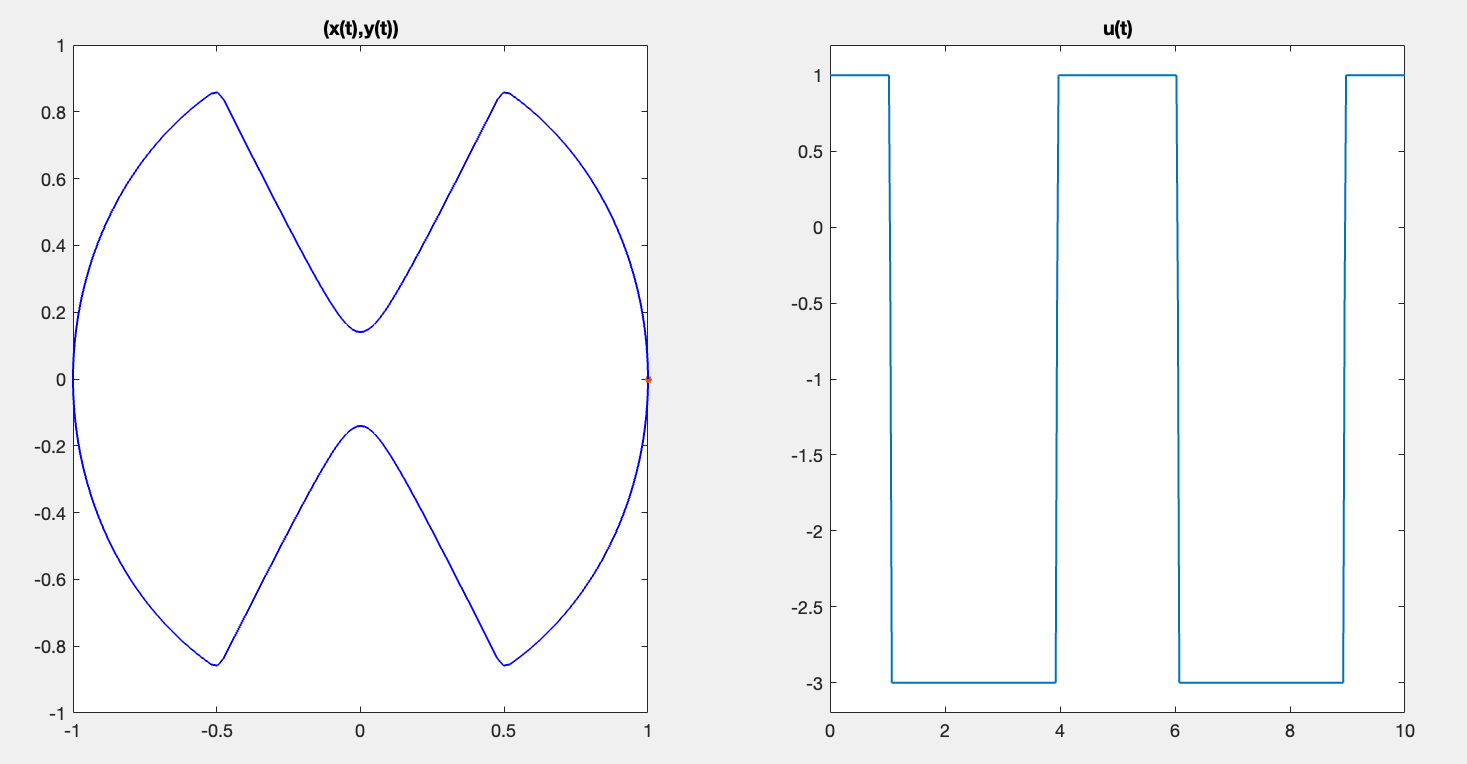}}
\caption{Numerical simulation of the butterfly solution, with $\umin=-3$, $\umax=1$, $T=10$.}\label{fig_butterfly}
\end{figure}

For $T=2\pi$, the butterfly trajectory is given by 
$$
x(t)=\cos(\omax t), 
\qquad y(t)=-\omax\sin(\omax t), 
\qquad u(t)=\umax ,
$$
and the curve is the ellipse $x^2+\frac{y^2}{\omax^2}=1$.
Actually for $T<2\pi$ there does not exist any solution crossing the $y$-axis, i.e., satisfying \eqref{OCP_cross}, so $2=2\pi$ is the minimal possible time for which butterfly solutions exist. 
Then, as $T\geq2\pi$ grows, the ellipse ``opens" at its minimal and maximal values of $y$, and one completes the curve with two arcs of hyperbole, as it can be seen on Figure \ref{fig_butterfly}. 
When $T\rightarrow+\infty$, these arcs of hyperbole tend to a ``V"-shape, with the edge of the ``V" tending to $(0,0)$: this is the turnpike phenomenon, similar to what has been discussed in Section \ref{sec_turnpike}.

\section{Applications in quantum and classical physics}
\subsection{Ground state of 1D Schr\"odinger operators with a finite potential well}

The physical motivation of the optimal control problem \eqref{OCP}-\eqref{OCP_0} lies in the study of one-dimensional Schr\"odinger operators with a finite potential well, given by
$$
P_{t_1,M} = -\frac{d^2}{dt^2} + V_{t_1,M}\, \mathrm{id}
$$
generally considered on the whole real line $\R$ (i.e., $T=+\infty$) or on an interval $\big(-\frac{T}{2},\frac{T}{2}\big)$, with suitable boundary conditions, where the potential $V_{t_1,M}$ is defined by
\begin{equation*}
V_{t_1,M}(t) = \left\{ \begin{array}{ll}
0 & \textrm{if}\ \ \vert t\vert\leq t_1, \\
M & \textrm{if}\ \ \vert t\vert>t_1, 
\end{array}\right.
\end{equation*}
depending on two parameters $M>0$ and $t_1>0$ that are respectively the height and the width of the potential well.
%When $M=+\infty$, this is a strictly confining potential.

More precisely, 
when $T<+\infty$ we consider $P_{t_1,M}$ on the domain
\begin{equation*}
\begin{split}
D(P_{t_1,M}) &= H^2_{\mathrm{per}}\Big(-\frac{T}{2},\frac{T}{2}\Big) \\
&= \bigg\{ \psi\in H^2\Big(-\frac{T}{2},\frac{T}{2}\Big) \quad \Big\vert\quad  \psi\Big(-\frac{T}{2}\Big)=\psi\Big(\frac{T}{2}\Big), \ \ \dot \psi\Big(-\frac{T}{2}\Big)=\dot \psi\Big(\frac{T}{2}\Big)  \bigg\} 
\end{split}
\end{equation*}
of periodic functions, and 
when $T=+\infty$ we consider the operator on the domain
$$
D(P_{t_1,M}) = \bigg\{ \psi\in H^2\Big(-\frac{T}{2},\frac{T}{2}\Big) \quad \Big\vert\quad \lim_{t\rightarrow-\infty}\psi(t)=\lim_{t\rightarrow+\infty}\dot \psi(t)=0 \bigg\} 
$$
of functions vanishing at infinity.

\medskip

The relationship with the optimal control problem \eqref{OCP}-\eqref{OCP_0} is the following. Let us consider any nontrivial optimal solution $(x(\cdot),y(\cdot),u(\cdot))$ of \eqref{OCP}-\eqref{OCP_per_0}.
As stated in Theorem \ref{main_thm}, they are all given by the same control \eqref{u_opt}, maybe shifted in time.
In Theorem \ref{main_thm}, as well as in its proof, for convenience we assumed that the trajectories were defined on the time interval $[0,T]$. Shifting in time (see \ref{A2} in Lemma \ref{lem_prelim}, let us now assume that they are defined on the time interval $\big(-\frac{T}{2},\frac{T}{2}\big)$. Then, now, the optimal control is given by
\begin{equation*}
u(t) = \left\{ \begin{array}{ll}
\umax & \textrm{if}\ \ \vert t\vert<t_1 , \\
\umin & \textrm{if}\ \ \vert t\vert>t_1 , \\
\end{array}\right.
\end{equation*}
for almost every $t\in\big(-\frac{T}{2},\frac{T}{2}\big)$.
Using \eqref{OCP_dyn_x}-\eqref{OCP_dyn_y}, we infer that, for $t\in\big(-\frac{T}{2},\frac{T}{2}\big)$,
\begin{align*}
- \frac{d^2}{dt^2} x(t) = \umax x(t) & \qquad\textrm{if}\ \ \vert t\vert < t_1  , \\
\Big(- \frac{d^2}{dt^2}+(\umax-\umin)\mathrm{id}\Big) x(t) = \umax x(t) & \qquad\textrm{if}\ \ \vert t\vert > t_1  .
\end{align*}
Therefore
$$
P_{t_1,M} x(\cdot) = \umax x(\cdot) 
\qquad\textrm{with}\qquad M=\umax-\umin = \omax^2+\omin^2.
$$
Since $x(\cdot)$ is nontrivial and %satisfies $x\big(-\frac{T}{2}\big)=x\big(\frac{T}{2}\big)$ and $\dot x\big(-\frac{T}{2}\big)=\dot x\big(-\frac{T}{2}\big)$, 
belongs to $D(P_{t_1,M})$,
this means that $x(\cdot)$ is an eigenfunction of $P_{t_1,M}$ associated with the eigenvalue $\umax$.
Actually it corresponds to the ground state of $P_{t_1,M}$, i.e., $\umax$ is the smallest eigenvalue of $P_{t_1,M}$.
%Noting that $\umax < M$ (because $\umin<0$), this shows that we compute here the eigenvalues of $P_{t_1,M}$ that belong on the interval $(0,M)$.

When $T\rightarrow+\infty$, we know, by Remark \ref{rem1}, that 
$$
x\Big(-\frac{T}{2}\Big)=x\Big(\frac{T}{2}\Big) \ \sim\ \frac{\omax}{\omin^2+\omax^2} \exp\bigg( \frac{\omin}{\omax} \mathrm{Arctan}\Big(\frac{\omin}{\omax}\Big) -\omin\frac{T}{2} \bigg) 
\longrightarrow 0
$$
and 
$$
t_1\longrightarrow t_1^\infty = \frac{1}{\omax}\mathrm{Arctan}\Big(\frac{\omin}{\omax}\Big) .
$$
This limit value of $t_1^\infty$ satisfies $\omax \tan(\omax t_1^\infty)=\omin$. It corresponds to well known computations in quantum physics, as explained next.

\paragraph{Computation of the ground state of $P_{t_1,M}$.}
Given a height $M>0$ and a width $t_1>0$, a well known problem in quantum physics consists of determining the ground state of of $P_{t_1,M}$ -- and more generally, the discrete eigenvalues of $P_{t_1,M}$ and their corresponding eigenfunctions.
Recall that the discrete eigenvalues of $P_{t_1,M}$ belong to the interval $(0,M)$ (beyond $M$, we have continuous spectrum).

We first explain this fact for $T=+\infty$, thus recovering classical computations in quantum physics (see, e.g., \cite{Hall}).

Let $\umax\in(0,M)$ be the lowest eigenvalue of $P_{t_1^\infty,M}$.
We set $\omin = \sqrt{M-\omax^2}$, so that $M=\umax-\umin$. Now, we must have
$$
t_1^\infty = \frac{1}{\omax}\mathrm{Arctan}\bigg(\frac{\sqrt{M-\omax^2}}{\omax}\bigg) .
$$
Since the function $\omax\mapsto t_1^\infty(\omax)$ is bijective on $(0,M)$, there exists a unique solution $\omax\in(0,M)$ of the above equation. 
In quantum physics, this solution is seen as the first solution of a quantification relation.
The corresponding eigenfunction (or ground state) is the function $x(\cdot)$ given by Theorem \ref{main_thm}. 

\begin{figure}[h]
%\begin{center}
%\resizebox{15cm}{!}{\input fig_hyperboles.pdf_t}
%\end{center}
\centerline{\includegraphics[width=\textwidth]{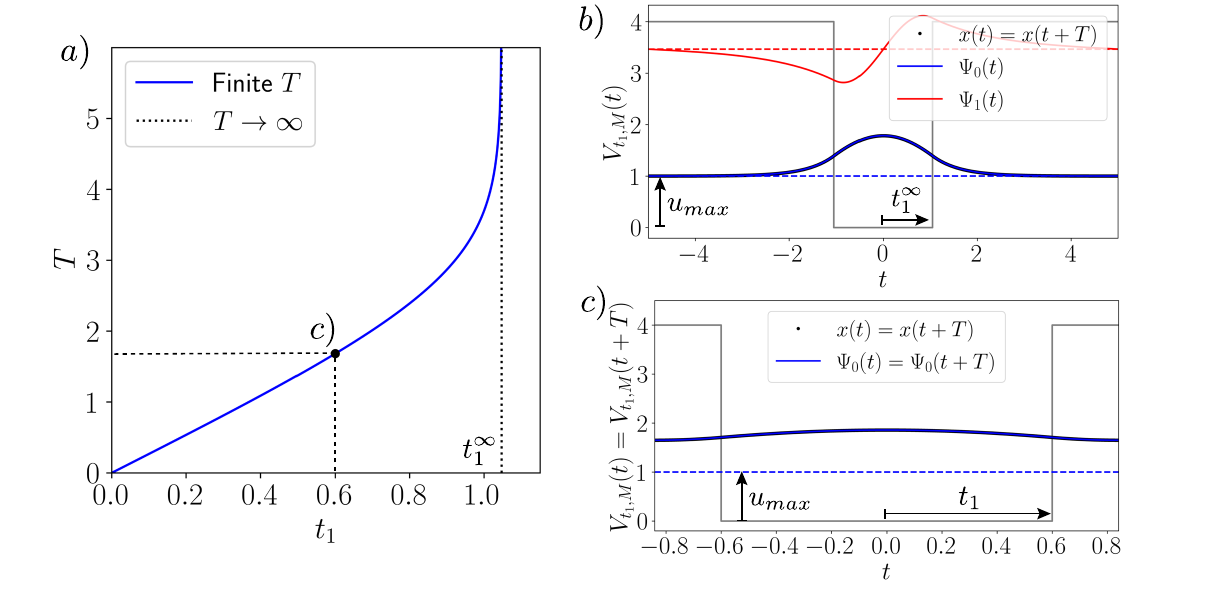}}
\caption{Example of the spectrum computation of a 1D Schr\"odinger operator with a finite potential well for $\umin=-3$ and $\umax=1$. \\
a) Evolution of $T$ as a function of $t_1$ for the optimal solution of \eqref{OCP}-\eqref{OCP_0} given in \eqref{Tt1}. \\
b) ``Bound states'' of a 1D Schr\"odinger operator with a finite potential with $M=\umax - \umin = 4$ and $t_1=t_1^{\infty}$. \\
c) Case with periodic boundary conditions and $M=\umax - \umin = 4$, $t_1=0.6$ and $T$ given in \eqref{Tt1} illustrated in a).}
\label{fig_Schrodinger}
\end{figure}

Figure \ref{fig_Schrodinger} shows an example for $\umax=1$ and $\umin=-3$, i.e., $M=4$. Figure \ref{fig_Schrodinger}a) shows the evolution of $T$ as a function of $t_1$ as given by \eqref{Tt1}. When $T=+\infty$, this function tends to $t_1^\infty$. Figure \ref{fig_Schrodinger}b) shows the operator $V_{t_1,M}$ for $T=+\infty$ ($T=10$ for numerical purposes), $t_1=t_1^\infty$ and $M=4$, as well as the two eigenvalues belonging to the interval $(0,M)$ in dashed lines and the two associated eigenfunctions $\Psi_0(t)$ and $\Psi_1(t)$ (the local $x$-axis of those eigenfunctions is located at the level of their eigenvalues). The eigenfunctions are normalized so that $\int_{-T/2}^{T/2}\Psi_n(t)^2\,dt=1$. As expected, the smallest eigenvalue is $\umax=1$ and the associated eigenfunction $\Psi_0$ corresponds to the solution $x(\cdot)$ of the optimal problem \eqref{OCP}-\eqref{OCP_0} with $\umax=1$, $\umin=-3$ and $T=+\infty$.

\medskip

When $T<+\infty$, the situation is a bit more complicated. Since $t_1$ is given in the quantum problem, the value of $T$ is determined by \eqref{Tt1} as shown in Figure  \ref{fig_Schrodinger}a). For $\umax=1$ and $\umin=-3$, the ($T$-periodic) solution $x(\cdot)$ of the optimal problem \eqref{OCP}-\eqref{OCP_0}, with $t_1=0.6$ for example, corresponds to a period $T \approx 1.682$ as shown on Figure  \ref{fig_Schrodinger}a). Looking for the eigenvalues of the operator $P_{t_1,M}$ with $M = \umax-\umin=4$, $t_1=0.6$, $T \approx 1.682$, we find only one eigenvalue $\umax\in(0,M)$ as shown on Figure  \ref{fig_Schrodinger}c). The $T$-periodic eigenfunction $\Psi_0$, normalized so that $\int_{-T/2}^{T/2}\Psi_n(t)^2\,dt=1$, is drawn on Figure  \ref{fig_Schrodinger}c) on the interval $[-T/2,T/2]$. This eigenfunction corresponds to the optimal solution $x(\cdot)$ of \eqref{OCP}-\eqref{OCP_0} computed with the same parameters.

\subsection{Optimal dynamical stabilization}

Another important application where the optimal control problem \eqref{OCP}-\eqref{OCP_0} matters is the so-called dynamical stabilization phenomenon, known in classical physics as the process by which charged particles can be trapped in periodically varying electromagnetic fields like in mass spectrometers or in trapped ion quantum computers (see \cite{Paul1990electromagnetic}). Dynamical stabilization consists of periodically modulating in time the properties of a system to dynamically sustain one of its naturally unstable configurational states. Recent works in this field (see \cite{Grandiprotierelazarus}) have shown a new modulation parameter regime that exhibits some mathematical analogy with the computation of ground states of 1D Schr\"odinger operators $P_{t_1,M}$ with potential, as described in the previous subsection. In this new framework, optimal control theory can therefore be used to determine the minimal modulation needed to stabilize a system, whence the wording of ``optimal dynamical stabilization".

\begin{figure}[h]
\centerline{\includegraphics[width=\textwidth]{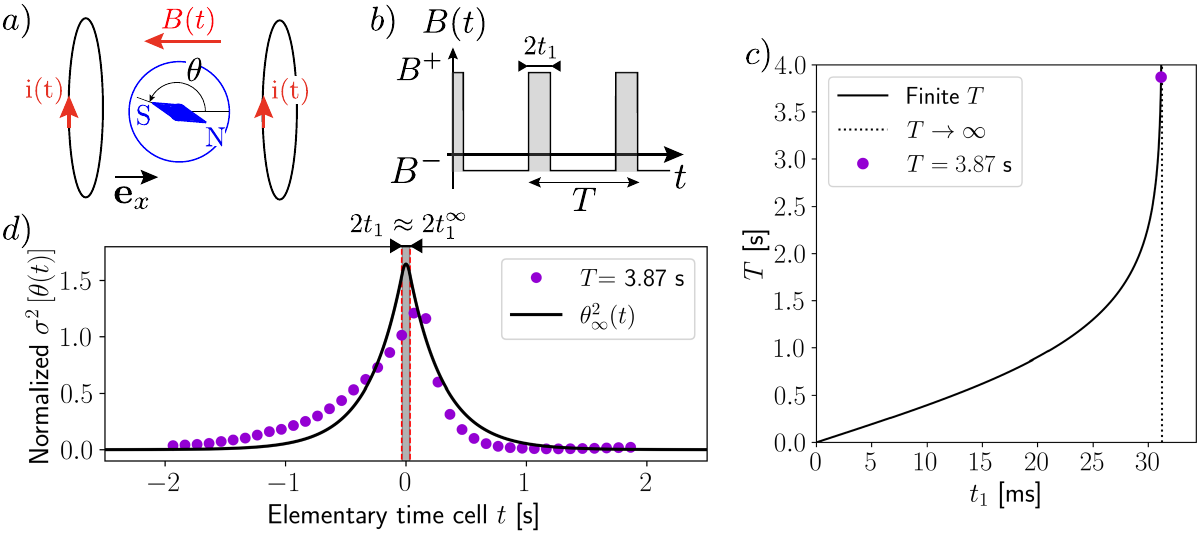}}
\caption{An example of experimental results in optimal dynamical stabilization (see \cite{LazarusTrelat_PNAS}). \\
a) A compass, fully parameterized by the angle $\theta(t)$ between $\mathbf{e}_x$ and its $N-S$ local axis is placed at the center of two Helmholtz coils oriented along $\mathbf{e}_x$. \\
b) We impose a $T$-periodic magnetic field $B(t)=B(t+T)$ with $B(t)=B^+=852$ $\mu$T ($i=-200$ mA) during a time $2t_1$ and $B(t)=B^-=-47$ $\mu$T ($i=0$ A) during $T-2t_1$. \\
c) Evolution of $T$ as a function of $t_1$ as given in \eqref{Tt1} for $\umax=B^+\mu/I=54.5$ (rad/s)$^2$ and $\umin=B^-\mu/I=-3$ (rad/s)$^2$. \\
d) Variance of the dynamically stable oscillatory motion $\theta(t)$ about $\theta=\pi$ for two experiments with $T=3.87$ s and $2t_1=70 \approx t_1^{\infty}=62$ ms. The square of the optimal solution of \eqref{OCP}-\eqref{OCP_0} for $T=\infty$, $\theta_{\infty}(t)$, is shown in black line.}
\label{fig_dynstab}
\end{figure}

The relationship between optimal dynamical stabilization and the optimal control problem \eqref{OCP}-\eqref{OCP_0} was discovered and exploited in \cite{LazarusTrelat_PNAS}. Figure \ref{fig_dynstab}a) is a sketch of the model experiment reported in \cite{LazarusTrelat_PNAS} that is an academic platform to explore the optimal dynamical stabilization of a one degree-of-freedom system in a periodically time-varying potential energy landscape. The experiment consists of a compass centered and aligned between two Helmholtz coils. The dipole's configurational state is fully parametrized in time by the angle $\theta(t)$ between the axis of the coils $\mathbf{e}_x$, coincident with the North-South magnetic axis of Earth, and the $N-S$ axis of the magnetized needle. This classical one degree-of-freedom nonlinear oscillator can be modeled by the nonlinear evolution equation 
\begin{equation}
\ddot{\theta}(t)+2\xi \sqrt{\frac{|B(i(t))|\mu}{I}}\dot{\theta}(t)+\frac{B(i(t))\mu}{I}\sin(\theta(t))=0 ,
\label{eq:compasNL}
\end{equation}
where $\xi=0.3$ is the experimental damping ratio, $\mu/I=6.4 \times 10^4$ A.kg$^{-1}$ is the ratio between the magnetic moment $\mu$ and moment of inertia $I$, and $\sqrt{|B(i(t))|\mu/I}$ is the natural frequency of the dipole around its stable equilibrium position that can be either $\theta(t)=0$ or $\theta(t)=\pi$ depending on the current $i(t)$ in the coils. Here, $B(i(t))=-(B_T + A\, i(t))$ is the magnitude of the uniform magnetic field felt by the dipole with $B_T=47$ $\mu$T (Earth magnetic field), $A=4496 $ $\mu$T/A (a property related to the Helmholtz coils configuration) and $i(t)$ is the current in the coils.

When $i(t)=0$, meaning that no power is supplied to the coils, the magnetic field points in the direction of $-\mathbf{e}_x$ and so does the North pole of the compass. In this case, $B(t)=B^-=-47<0$, making $\theta =0$ a stable equilibrium and $\theta=\pi$ unstable. A compass starting from any initial condition will eventually converge to $\theta = 0$ with damped oscillations, moving away from $\theta = \pi$. A natural physical question is: 
\begin{quote}
\textit{What is the minimal current over time, or more precisely the minimal value of $\int |i(t)| \, dt$, required to stabilize $\theta = \pi$?}
\end{quote}
If we ignore time and only consider constant currents, the answer is straightforward: in the experiment, for any constant current $i(t)=i < -10$ mA, the stability of the equilibrium configurations is reversed, and $\theta=\pi$ becomes asymptotically stable. The minimal value of $\int |i(t)| \, dt$ is at least $10$ mA multiplied by the total time. However, when time modulations are allowed, for example with a periodic current $i(t)=i(t+T)$, it becomes possible to reduce the value of $\int |i(t)| \, dt$, even more by switching off the coils ($i(t)=0$) during part of the cycle, while still maintaining dynamic stability. The problem of minimizing $\int |i(t)| \, dt$ by increasing the duration of $i(t)=0$ while keeping the periodicity condition $i(t)=i(t+T)$ can be framed as the optimal control problem \eqref{OCP}-\eqref{OCP_0}.

\medskip

Figure \ref{fig_dynstab}b) shows an example of a typical periodic current applied in the coils, leading to a positive magnetic field $B(t)=B^+=+852$ $\mu$T when $i(t)=-0.2$ A ($\theta = \pi$ is an attractor) and a negative magnetic field $B(t)=B^-=-47$ $\mu$T when $i(t)=0$ A ($\theta = \pi$ is a repeller). The experiment reported in \cite{LazarusTrelat_PNAS} consists of seeking, in the $(t_1,T)$ modulation parameter space, dynamically stable responses $\theta(t)$, remaining in the vicinity of $\theta=\pi$ for initial conditions $\theta(0) \approx \pi$ and $\dot \theta(0)=0$. This stability can be rationalized by linearizing \eqref{eq:compasNL} about $\theta=\pi$, leading to 
\begin{equation}
\ddot{\theta}(t)+u(t)\theta(t)=0
\label{eq:compasL}
\end{equation}
where $u(t)=u(t+T)=\umax=B^+\mu/I=54.5$ (rad/s)$^2$ during $2t_1$ and $u(t)=\umin=B^-\mu/I=-3$ (rad/s)$^2$ during $T-2t_1$ (since the compass is only doing about half an oscillation for the considered $t_1$, the damping factor can be neglected in a first approximation). By Floquet theory, the solution of \eqref{eq:compasL} can be expressed as $\theta(t)=\Psi(t)e^{st}+\bar{\Psi}(t)e^{-st}$ where $\Psi$ is a complex $T$-periodic function and $s\in\C$. In the $(t_1,T)$ modulation parameter space, one observe an alternance of unstable ($\Re(s)>0$) and dynamically stable ($\Re(s)=0$) tongues (see \cite{LazarusTrelat_PNAS}). The dynamically stable tongues are bounded by $T$ and $2T$ periodic solutions $\theta$, that become narrower as $T$ increases (see \cite{Grandiprotierelazarus,LazarusTrelat_PNAS}).

The $T$-periodic solutions $\theta$ of \eqref{eq:compasL} as well as their associated $T$-periodic modulation function $u$, correspond to the solutions of the optimal control problem \eqref{OCP}-\eqref{OCP_0} (and the lower boundary of the first stability tongue). Notably, they minimize $\int_0^T u(t) \, dt$, i.e., $\int |i(t)| \, dt$, given that $\umin <0$ corresponds to $i(t)=0$ in the experimental setup. When $T \rightarrow +\infty$, one should theoretically be able to dynamically stabilize the compass in $\theta=\pi$ with almost no current! Figure \ref{fig_dynstab}c) shows the evolution of $T$ as a function of $t_1$ according to \eqref{Tt1}  for the optimal $u(\cdot)$ with $\umax=54.5$ and $\umin=-3$. When $T=+\infty$, the optimal control $u(\cdot)$ has a constant duration $2t_1^{\infty}\approx 62$ ms, and a periodic solution, denoted $\theta_{\infty}(\cdot)$, should exist (teardrop optimal trajectory in the state space). In practice, by experimentally imposing $2t_1=70$ ms of $B(t)=B^+$ (the dissipation slightly switches the stability regions in the modulation parameter space), it has been possible to turn off the coils during $T-2t_1=3.8$ s, i.e., more than $98$\% of the time as shown on Figure \ref{fig_dynstab}d). Beyond this period of $T=3.87$ s, the compass is no more stable in practice because of the inherent imperfections of the setup and of the basin of attraction of initial conditions that shrinks about $\theta(0) = \dot \theta(0) = 0$ (see \cite{Grandiprotierelazarus}).

The theoretical optimal solution $\theta_{\infty}(\cdot)$, that is almost a teardrop in the state space for $T=3.87$ s, should predict the periodic oscillations of the compass about $\theta=\pi$. However, due to the nature of optimal dynamical stabilization for large $T$ that is a symmetry breaking that repeats periodically, the observed experimental oscillations are actually quasi-periodic, consisting of a succession of scaled functions $\theta_{\infty}(\cdot)$ periods after periods (see \cite{LazarusTrelat_PNAS}). Moreover, the experiments being sensitive to initial conditions, various experiments with the seemingly same parameters do not lead to the same quasi-periodic responses $\theta(t)$. \textit{In fine}, it is not the oscillations $\theta(\cdot)$ that are well predicted over time by the optimal control problem \eqref{OCP}-\eqref{OCP_0}, but rather their variance $\sigma^2[\theta(t)]$ over one period, that are consistently predicted by $\theta^2_{\infty}(t)$ as shown on Figure \ref{fig_dynstab}d).

\end{document}